\documentclass[10pt]{article}

\usepackage{amsmath}    % need for subequations
\usepackage{graphicx}   % need for figures
\usepackage{verbatim}   % useful for program listings
\usepackage{color}      % use if color is used in text
\usepackage{hyperref}   % use for hypertext links, including those to external documents and URLs

\usepackage{amssymb}
\usepackage{amsthm}
\usepackage{mathrsfs}

\def\Xint#1{\mathchoice
	{\XXint\displaystyle\textstyle{#1}}%
	{\XXint\textstyle\scriptstyle{#1}}%
	{\XXint\scriptstyle\scriptscriptstyle{#1}}%
	{\XXint\scriptscriptstyle\scriptscriptstyle{#1}}%
	\!\int}
\def\XXint#1#2#3{{\setbox0=\hbox{$#1{#2#3}{\int}$}
		\vcenter{\hbox{$#2#3$}}\kern-.5\wd0}}

\def\dashint{\Xint-}
\newtheorem{Def}{Definition}[section]
\newtheorem{Eg}[Def]{Example}

\newtheorem{Lem}[Def]{Lemma}
\newtheorem{Thm}[Def]{Theorem}
\newtheorem{Cor}[Def]{Corollary}
\newtheorem{Ass}[Def]{Assumption}

\theoremstyle{definition}
\newtheorem{Rem}[Def]{Remark}

\newcommand{\p}{\mathbb{P}}
\newcommand{\e}{\mathbb{E}}
\newcommand{\real}{\mathbb{R}}
\newcommand{\n}{\mathbb{N}}

\newcommand{\1}{{\bf 1}}

\newcommand{\rd}{\mathrm{d}}

\allowdisplaybreaks

\begin{document}

\title{
$L^{q}$-error estimates for approximation of irregular functionals of random vectors
%--Multi-dimensional Avikainen's estimates
%Multi-dimensional Avikainen's estimates -- irregular functionals of random variables
}
\author{
	Dai Taguchi\footnote{
		Research Institute for Interdisciplinary Science, department of mathematics,
		Okayama University,
		3-1-1
		Tsushima-naka,
		Kita-ku
		Okayama
		700-8530,
		Japan,
		email~:~\texttt{dai.taguchi.dai@gmail.com}
	}
	,~
	Akihiro Tanaka\footnote{
		Graduate School of Engineering Science,
		Osaka University,
		1-3, Machikaneyama-cho, Toyonaka,
		Osaka, 560-8531, Japan
		/
		Sumitomo Mitsui Banking Corporation,
		1-1-2, Marunouchi, Chiyoda-ku,
		Tokyo, 100-0005, Japan,
		email~:~\texttt{tnkaki2000@gmail.com}
	}
	~and~
	Tomooki Yuasa\footnote{
		Department of Mathematical Sciences,
		Ritsumeikan University,
		1-1-1,
		Nojihigashi, Kusatsu,
		Shiga, 525-8577, Japan,
		email~:~\texttt{to-yuasa@fc.ritsumei.ac.jp}
	}
}
%\date{}
\maketitle
\begin{abstract}
	In \cite{Av09}, Avikainen showed that, for any  $p,q \in [1,\infty)$, and any function $f$ of bounded variation in $\mathbb{R}$, it holds that 
	$
	\mathbb{E}[|f(X)-f(\widehat{X})|^{q}]
	\leq
	C(p,q)
	\mathbb{E}[|X-\widehat{X}|^{p}]^{\frac{1}{p+1}}
	$, where $X$ is a one-dimensional random variable with a bounded density, and  $\widehat{X}$ is an arbitrary  random variable.
	In this article, we will provide multi-dimensional versions of this estimate for functions of bounded variation in $\mathbb{R}^{d}$, Orlicz--Sobolev spaces, Sobolev spaces with variable exponents, and fractional Sobolev spaces.
	The main idea of our arguments is to use the Hardy--Littlewood maximal estimates and pointwise characterizations of these function spaces.
	We apply our main results to analyze the numerical approximation for some irregular functionals of the solution of stochastic differential equations.
	%, and to estimate the $L^{2}$-time regularity of decoupled forward--backward stochastic differential equations with irregular terminal conditions.
	%numerical schemes for forward--backward stochastic differential equations based on machine learning.
	%65C05 Monte Carlo methods
	%26A45 Functions of bounded variation, generalizations
	%46E35 Sobolev spaces and other spaces of \smooth" functions, embedding theorems, trace theorems
	%65C30:Numerical solutions to stochastic differential and integral equations
	%60H35 Computational methods for stochastic equations
	%41A25 Rate of convergence, degree of approximation
	%91G60 Numerical methods (including Monte Carlo methods)
	%65Cxx Probabilistic methods, simulation and stochastic dierential equations For theoretical aspects, see 68U20 and 60H35
	%62P05 Applications to actuarial sciences and nancial mathematics
\end{abstract}

%\tableofcontents

\section{Introduction}\label{sec_1}

Numerical analysis for stochastic differential equations (SDEs) is a  significant issue in stochastic calculus, from both theory and practical points of view.
In order to approximate the solution of the SDE $\rd X(t)=b(X(t)) \rd t+\sigma(X(t))\rd B(t)$ driven by a Brownian motion $B$, one often uses the Euler--Maruyama scheme $X^{(h)}$ with time step $h>0$ (see, \cite{KP}).
The convergence of this scheme  has been widely studied, see, e.g.,  \cite{BaTa96, GoLa08, Gu06, KoMa02, KoMe17, TaTu90} for the  weak convergence;  \cite{Av09,BaHuYu19,BuDaGe19,GyRa11,LeSz17b,MeTa,MuYa20,NT1,NT2,Ta20} for the strong convergence.
There are many modifications of the Euler--Maruyama scheme, see, e.g, \cite{Al13,HiMaSt02,NeSz14} for backward schemes and \cite{HuJeKl12,Sa16} for tamed schemes.
In particular, Bally and Talay \cite{BaTa96} proved that whenever the coefficients $b$ and $\sigma$ satisfy some certain regularity conditions, $|\e[f(X(T))]-\e[f(X^{(h)}(T))]|\leq Ch$ for any bounded measurable function $f$, and time step $h=T/n$, (see, also \cite{GoLa08,Gu06,KoMa02,TaTu90}).

In this article, for an irregular function $f$ (e.g., bounded variation or Sobolev differentiable) and $q \in [1,\infty)$, we are interested in the rate of convergence to zero of
\begin{align}\label{MSE_0}
	\e\left[
		\left|
			f(X(T))
			-
			f(X^{(h)}(T))
		\right|^{q}
	\right]
\end{align}
as $h \to 0$ (see, Theorem \ref{Cor_0}).
This rate is needed to apply the multilevel Monte Carlo  (MLMC) method, whose computational cost is much lower than that of classical (single level) Monte Carlo method, for $\e\left[\left|f(X(T))\right|\right]$.
Heinrich \cite{He01} firstly introduced the MLMC method for parametric integrations.
Later, Giles \cite{Gi08} developed the method for SDEs based on the Euler--Maruyama scheme $X^{(h)}$ as a generalization of the statistical Romberg method proposed by Kebaier \cite{Ke05}.
%The MLMC method can reduce the computational complexity for the mean squared error $\e[|\widehat{Y}-\e[f(X(T)]|^2]$, where $\widehat{Y}$ is an estimator for $f(X(T))$, and requires the rate of convergence for \eqref{MSE_0} with $q=2$.
If the function $f$ is $\alpha$-H\"older continuous for some $\alpha \in (0,1]$, then \eqref{MSE_0} can be bounded from above by $\|f\|_{\alpha}^{q} \e[|X(T)-X^{(h)}(T)|^{q \alpha}]$, where $\|f\|_{\alpha}:=\sup_{x \neq y} \frac{|f(x)-f(y)|}{|x-y|^{\alpha}}$, and thus its  rate of convergence to zero can be obtained under some suitable assumptions on the coefficients $b$ and $\sigma$ (see, e.g., \cite{BaHuYu19, BuDaGe19,GyRa11,KP,LeSz17b,MeTa,MuYa20,NT1,NT2}).
However, if the function $f$ is irregular, it is not clear how to derive the rate of convergence to zero.

Motivated by such a problem, Avikainen \cite{Av09} proved the following remarkable inequality.
\begin{Thm}[Theorem 2.4 (i) in \cite{Av09}]
	Let $X$ be a one-dimensional random variable with a bounded density function $p_{X}$, and let $f$ be a real valued function of bounded variation in $\real$.
	Then for any  $p,q \in [1,\infty)$ and random variable $\widehat{X}$, it holds that
	\begin{align}\label{Av_0}
		\e\left[
			\left|
				f(X)
				-
				f(\widehat{X})
			\right|^{q}
		\right]
		\leq
		3^{q+1}
		V(f)^{q}
		\left(
			\sup_{x \in \real}
			p_{X}(x)
		\right)^{\frac{p}{p+1}}
		\e
		\left[
			\left|
				X
				-
				\widehat{X}
			\right|^{p}
		\right]^{\frac{1}{p+1}},
	\end{align}
	where $V(f)$ is the total variation of $f$.
\end{Thm}
Note that this estimate is optimal, that is, there exist some random variables $X$, $\widehat{X}$ and function $f$ such that the equality holds in \eqref{Av_0} (see, Theorem 2.4 (ii) in \cite{Av09}).
The proof of this estimate is based on the following three ideas.
The first idea is to use the Lipschitz continuity of the distribution function of $X$, which is equivalent to the existence of a bounded density function of $X$.
Note that if the distribution function of $X$ is H\"older continuous, it is still able to obtain  a generalized version of Avikainen's estimate, which  is useful to evaluate  numerical schemes for some stochastic processes (see, \cite{Ta20}).
The second idea is to use Skorokhod's ``explicit" representation to embed the distribution of $X$ in the probability space $([0,1], \mathscr{B}([0,1]), \mathrm{Leb})$ (e.g., Section 3 in \cite{Wi91}).
For multi-dimensional random variables, this representation is known as Skorokhod's embedding theorem (e.g., Theorem 2.7 in \cite{IkWa}).
However, it might be difficult to apply it to the multi-dimensional case since it is not explicit.
The third idea is that every function $f$ of (normalized) bounded variation in $\real$ can be expressed as an integral of indicator functions $\1_{(z,\infty)}$ with respect to a signed measure $\nu(\rd z)$ which has a  bounded variation.
Then the estimate \eqref{Av_0} can be obtained by first considering it for $f(x)=\1_{(z,\infty)}(x)$ (see, Lemma 3.4 in \cite{Av09} for details and Proposition 5.3 in \cite{GiXi17} for a simple proof).

In this article, we will propose some versions of Avikainen's estimate \eqref{Av_0} for multi-dimensional random variables.
To the best of our knowledge, there is no result in this direction so far.
As mentioned above, it might be difficult to apply the approach in \cite{Av09} for multi-dimensional random variables.
Instead, we propose a new approach based on the Hardy--Littlewood maximal operator $M$ for locally finite vector valued measures $\nu$, which is defined by
\begin{align*}
	M\nu(x)
	:=
	\sup_{s>0}
	\dashint_{B(x;s)}
	\rd |\nu|(z),
	~
	\dashint_{B(x;s)}
	\rd |\nu|(z)
	:=
	\frac{|\nu|(B(x;s))}{\mathrm{Leb}(B(x;s))},~
	x \in \real^{d},
\end{align*}
where $|\nu|$ is the total variation of $\nu$ and $B(x;r)$ is the closed ball in ${\mathbb R}^{d}$ with center $x$ and radius $r$.
The operator $M$ is well-studied in the fields of real analysis and harmonic analysis, and it satisfies the following Hardy--Littlewood maximal weak type estimate
\begin{align*}
	\mathrm{Leb}
	(\{x\in \real^{d}~;~M\nu(x) > \lambda\})
	\leq
	A_{1}
	|\nu|(\real^{d})
	\lambda^{-1},~\lambda>0,
\end{align*}
where the constant $A_{1}$ depends only on $d$.
Using this estimate, we will prove that for any random variables $X,\widehat{X}:\Omega \to \real^{d}$ with density functions $p_{X}$ and $p_{\widehat{X}}$ with respect to Lebesgue measure, respectively, and for any $f \in BV(\real^{d}) \cap L^{\infty}(\real^{d})$, $p \in (0, \infty)$ and $q\in [1,\infty)$, if $p_{X}$ and $p_{\widehat{X}}$ are bounded, then it holds that
\begin{align}\label{Av_1}
	\e\left[
		\left|
			f(X)
			-
			f(\widehat{X})
		\right|^{q}
	\right]
	\leq
	C
	\e\left[
		\left|
			X
			-
			\widehat{X}
		\right|^{p}
	\right]^{\frac{1}{p+1}}
\end{align}
for some constant $C$ which depends on $p,q,d,f$, $\|p_{X}\|_{\infty}$ and $\|p_{\widehat{X}}\|_{\infty}$ (for more details, see, Theorem \ref{main_0}).
Here, $BV(\real^{d})$ is the class of functions $f$ of bounded variation in $\real^{d}$, which is a subset of $L^{1}(\real^{d})$ such that the total variation $|Df|(\real^{d})=\int_{\real^{d}} |Df|$ of the Radon measure $Df$ defined by
\begin{align*}
	\int_{\real^{d}}|Df|
	:=
	\sup
	\left\{
		\int_{\real^{d}}
			f(x) \mathrm{div} g(x)
		\rd x
		~;~
		g \in C^{1}_{\mathrm{c}}(\real^{d};\real^{d})
		\text{~and~}
		\sup_{x \in \real^{d}}
		|g(x)|
		\leq 1
	\right\}
\end{align*}
is finite, where the Radon measure $Df$ is defined as the generalized derivative formulated by the integration by parts for functions of bounded variations (for more details, see, Section \ref{sec_2_1}).
The most important property of $f \in BV(\real^{d})$ which we use in this article is the following pointwise estimate (see, Lemma \ref{Lem_key_0})
%and Remark \ref{Rem_0})
\begin{align}\label{pw_esti}
	|f(x)-f(y)|
	&\leq
	K_{0}
	|x-y|
	\left\{
		M_{2|x-y|}(Df)(x)
		+
		M_{2|x-y|}(Df)(y)
	\right\},
	\text{ $\mathrm{Leb}$-a.e. }x,y \in \real^{d},
\end{align}
where for $R>0$, $M_{R}\nu$ is the restricted Hardy--Littlewood maximal function defined by
\begin{align*}
	M_{R}\nu(x)
	:=
	\sup_{0<s\leq R}
	\dashint_{B(x;s)}
	\rd |\nu|(z),~
	x \in {\mathbb R}^{d}.
\end{align*}
It is worth noting that Haj\l{}asz \cite{Ha96,Ha03} characterized Sobolev spaces $W^{1,p}(\real^{d})$, $1 \leq p <\infty$ by using a pointwise estimate similar to \eqref{pw_esti}, and defined Sobolev spaces on metric spaces using its pointwise estimate.
On the other hand, Lahti and Tuominen \cite{LaTu14}, and Tuominen \cite{Tu07} generalized this characterization to $BV(\real^{d})$ and the Orlicz--Sobolev space $W^{1,\Phi}(\real^{d})$ with a Young function $\Phi$ such that both $\Phi$ and its complementary function $\Psi$ satisfy the $\Delta_{2}$-condition (or are doubling, see Section \ref{sec_2_1}).
Moreover, the Sobolev space $W^{1,p(\cdot)}(\real^{d})$ with a variable exponent $p:\real^{d} \to [1,\infty]$ and the fractional Sobolev space $W^{s,p}(\real^{d})$ for $(s,p) \in (0,1) \times [1,\infty)$ also satisfy some pointwise estimates similar to \eqref{pw_esti}.
Inspired by these facts, we will also provide estimates similar to \eqref{Av_1} for $f \in \{W^{1,\Phi}(\real^{d}) \cup W^{1,p(\cdot)}(\real^{d}) \cup W^{s,p}(\real^{d})\} \cap L^{\infty}(\real^{d})$.
%Note that our approach to use Hardy--Littlewood maximal estimates and pointwise estimates is  also applicable to the one-dimensional case.

%Finally, we present another application of Avikanen's estimate.
%Recently several numerical schemes based on machine learning for high-dimensional forward--backward stochastic differential equations (FBSDEs) have been studied (e.g., \cite{EHaJe17,HaJeE18,HuPhWa20}).
%In particular, Hur\'e, Pham and Warin \cite{HuPhWa20} proposed numerical schemes based on a backward dynamic programming equation, and they provided an upper bound of their squared error.
%By applying the estimate \eqref{Av_1} to its upper bound, we can ensure the convergence of their numerical schemes for solutions to high-dimensional FBSDEs with irregular terminal conditions (see, Theorem \ref{main_5} for the $L^{2}$-time regularity of BSDEs and Remark \ref{Rem_BSDE}).

This article is structured as follows.
In Section \ref{sec_2}, we first recall the definitions and properties of functions of bounded variation in $\real^{d}$, Orlicz--Sobolev spaces, Sobolev spaces with variable exponents, fractional Sobolev spaces and the Hardy--Littlewood maximal function, and recall its estimates on their function spaces.
Then we will provide multi-dimensional versions of Avikainen's estimate (see, Theorem \ref{main_0} \ref{main_1}, \ref{main_2}, \ref{main_3}) for these function spaces.
In Section \ref{sec_3}, we apply our main results to numerical analysis on irregular functionals of the solution of SDEs based on the Euler--Maruyama scheme and the multilevel Monte Carlo method.
%, and to estimates of the $L^{2}$-time regularity of decoupled FBSDEs with irregular terminal conditions.

\subsection*{Notations}
We give some basic notations and definitions used throughout this article.
We consider that elements of $\real^d$ are column vectors, and for $x\in\real^d$, we write $x=(x_{1},\ldots,x_{d})^{\top}$.
Let $C^{1}_{\mathrm{c}}(U;\real^q)$ be the space of $\real^q$-valued functions on an open set $U$ of $\real^{d}$ with compact support such that the first continuous partial derivatives on $U$ exist.
For differentiable functions $f:\real^{d} \to \real^{d}$ and $g:\real^{d} \to \real$, we define the divergence of $f$ by $\mathrm{div}f:=\sum_{i=1}^{d}\frac{\partial f_{i}}{\partial x_{i}}$ and the gradient of $g$ by $\nabla g=(\frac{\partial g}{\partial x_{1}},\ldots,\frac{\partial g}{\partial x_{d}})^{\top}$.
For an essentially bounded measurable function $f:\real^{d} \to \real^{q}$, the supremum norm of $f$ is defined by $\|f\|_{\infty}:={\rm ess~sup}_{x \in \real^{d}}|f(x)|$.
For measurable functions $f:\real^{d} \to \real$ and $g:\real^{d} \to [0,\infty)$, we define
\begin{align*}
	\|f\|_{L^{r}(\real^{d}, g)}
	:=
	\left\{ \begin{array}{ll}
		\displaystyle
		\left(
			\int_{\real^{d}}
				|f(x)|^{r}
				g(x)
			\rd x
		\right)^{1/r},
		&\text{ if } r \in [1,\infty),  \\
		\displaystyle 
			\|f\|_{\infty},
		&\text{ if } r = \infty
	\end{array}\right.
\end{align*}
and the class of all functions with $\|f\|_{L^{r}(\real^{d}, g)}<\infty$ by $L^{r}(\real^{d}, g)$.
In particular, if $g\equiv 1$, then we use the notation $L^{r}(\real^{d})$ as usual $L^{r}$ space in $\real^{d}$.
For $s>0$ and $x \in \real^{d}$, we denote open and closed balls by $U(x;s):=\{y \in \real^{d}~;~|y-x| < s\}$ and $B(x;s):=\{y \in \real^{d}~;~|y-x|\leq s\}$, respectively.
For an invertible $d \times d$-matrix $A=(A_{i,j})_{1 \leq i,j\leq d}$, we set $|A|^2:=\sum_{i,j=1}^{d} A_{i,j}^2$ and
$g_A(x,y)
=\frac{\exp(-\frac{1}{2} \langle A^{-1}(y-x),y-x \rangle_{{\mathbb R}^{d}})}
{(2\pi)^{d/2} \sqrt{\det A}}$,
and $g_{c}(x,y)=g_{cI}(x,y)$ for $c>0$, where the matrix $I$ is the identity matrix.
We denote the gamma function by $\Gamma(x):=\int_{0}^{\infty} t^{x-1}e^{-t} \rd t$ for $x \in (0,\infty)$.

\section{Multi-dimensional Avikainen's estimates}\label{sec_2}

Let $(\Omega,\mathscr{F},\p)$ be a probability space and let $d$ be a positive integer.
In this section, we provide multi-dimensional versions of Avikainen's estimate for any random variables $X$ and $\widehat{X}$ with bounded density functions and for functions of bounded variation in $\real^{d}$, Orlicz--Sobolev spaces, Sobolev spaces with variable exponents and fractional Sobolev spaces.

\subsection{Function spaces}\label{sec_2_1}
In this subsection, we provide the definitions and properties of functions of bounded variation in $\real^{d}$, Orlicz--Sobolev spaces, Sobolev spaces with variable exponents and fractional Sobolev spaces.

\subsubsection*{Bounded variation in $\real^{d}$}
We first recall the definition of functions of  bounded variation in an open subset $U$ of $\real^{d}$.
For more detail, we refer to \cite{EvGa92,Gi84}.
A function $f \in L^{1}(U)$ has bounded variation in $U$, denoted by $f \in BV(U)$, if
\begin{align*}
	\int_{U}|Df|
	:=
	\sup
	\left\{
		\int_{U}
		f(x) \mathrm{div} g(x)
		\rd x
		~;~
		g \in C^{1}_{\mathrm{c}}(U;\real^{d}),~
		\sup_{x \in U}
		|g(x)|
		\leq 1
	\right\}
	<\infty.
\end{align*}
We call $\int_{U}|Df|$ the total variation of $f$ in $U$.
A function $f \in L_{\mathrm{loc}}^{1}(U)$ has locally bounded variation in $U$, denoted by $f \in BV_{\mathrm{loc}}(U)$, if $\int_{V}|Df|<\infty$ for any open set $V \subset U$ such that its closer $\overline{V}$ is compact and $\overline{V} \subset U$.

It follows from the structure theorem (e.g., Theorem 5.1 in \cite{EvGa92}) that for $f \in BV_{\mathrm{loc}}(U)$, there exists a vector valued Radon measure $Df=(D_{1}f,\ldots,D_{d}f)^{\top}$ on $(U,\mathscr{B}(U))$ such that the following integration by parts formula holds:
\begin{align*}
	\int_{U}
		f (x)\mathrm{div} g(x)
	\rd x
	=
	-
	\sum_{k=1}^{d}
	\int_{U}
		g_{k}(x)
	D_{k}f(\rd x)
	,~\text{for all}~g \in C^{1}_{\mathrm{c}}(U;\real^{d}).
\end{align*}

%We say a Borel subset $E$ of $\real^{d}$ has locally finite perimeter in $U$ if $\1_{E} \in BV_{\mathrm{loc}}(U)$, and we say $E$ is a Caccioppoli set if it has locally finite perimeter in every bounded open subset $U$ of ${\mathbb R}^{d}$.

\begin{Rem}\label{Rem_BV_0}
	\begin{itemize}
		\item[(i)]
		$BV(U)$ is a Banach space with the norm $\|f\|_{BV(U)}:=\|f\|_{L^{1}(U)}+\int_{U}|Df|$ (see, Remark 1.12 in \cite{Gi84}).
		
		\item[(ii)]
		Let $\{f_{n}\}_{n \in \n}$ be a sequence of functions in $BV(U)$ which converges to $f$ in $L^{1}_{\mathrm{loc}}(U)$.
		Then it holds that $\int_{U}|Df| \leq \liminf_{n \to \infty} \int_{U} |Df_{n}|$ (semi-continuity, e.g., Theorem 5.2 in \cite{EvGa92} or Theorem 1.9 in \cite{Gi84}).
		
		\item[(iii)]
		Sobolev's inequality holds on $BV(\real^{d})$, that is, if $f \in BV(\real^{d})$ and $d \geq 2$, then there exists $C>0$ such that $\|f\|_{L^{d/(d-1)}({\mathbb R}^{d})} \leq C\int_{\real^{d}}|Df|$, (see, Theorem 5.10 (i) in \cite{EvGa92} or Theorem 1.28 (A) in \cite{Gi84}).
		And if $d=1$, then $\|f\|_{\infty} \leq \int_{{\mathbb R}}|Df|$.
		Indeed, in the same way as the proof of Theorem 5.6 in \cite{EvGa92}, we choose $f_{k} \in C^{1}_{\mathrm{c}}(\real;\real)$, $k \in \n$ such that $f_{k} \to f$ $\mathrm{Leb}$-a.e. and $\int_{\real}|f'_{k}(z)| \rd z \to \int_{\real} |Df|$ as $k \to \infty$.
		Then by the fundamental theorem of calculus, $|f_{k}(x)| \leq |\int_{-\infty}^{x}f_{k}'(z) \rd z| \leq \int_{\real}|f_{k}'(z)| \rd z$, which implies $\|f\|_{\infty} \leq \int_{{\mathbb R}}|Df|$.
		
		%\item[(iv)]
		%In the theory of stochastic calculus, Caccioppoli sets are related to reflected Brownian motions (see, \cite{ChFiWi93,Fu97,Fu99}).
	\end{itemize}
\end{Rem}

\begin{Eg}\label{Eg_0}
	\begin{itemize}
		\item[(i)]
		Let $W^{1,1}(U)$ be a Sobolev space.
		Then $W^{1,1}(U) \subset BV(U)$ (see, e.g., Example, page 197 of \cite{EvGa92} or Example 1.2 in \cite{Gi84}).
		
		\item[(ii)]
		Let $E$ be a bounded subset of $\real^{d}$ with $C^{2}$ boundary.
		Then $\1_{E} \in BV(\real^{d}) \setminus W^{1,1}(\real^{d})$ (see, e.g. Example 1.4 in \cite{Gi84}).
		%In this case, it holds that the Gauss--Green formula
		%\begin{align*}
		%	\int_{E}
		%		\mathrm{div} g(x)
		%	\rd x
		%	=
		%	\int_{\partial E}
		%		\langle
		%			g(x),\nu_{E}(x)
		%		\rangle_{\real^{d}}
		%	H^{d-1}(\rd x),~\text{for all}~g \in C^{1}_{\mathrm{c}}(U;\real^{d})
		%\end{align*}
		%and $\int_{U}|D\1_{E}|=H^{d-1}(U \cap \partial E)$, where $\nu_{E}(x)$ is the outward unit normal vector to $\partial E$ at $x$ and $H^{d-1}$ is the $(d-1)$-dimensional Hausdorff measure (see, e.g., Example 1.4 in \cite{Gi84}).
		%Moreover, De Giorgi's structure theorem (e.g.,  Theorem 15.9 in \cite{Ma12} or Theorem 5.16 in \cite{EvGa92}) shows the generalized Gauss--Green formula for Caccioppoli sets or locally finite perimeter sets in $\real^{d}$.
		
		%\item[(iii)]
		%Let $f \in BV(\real^{d})$ and $E_{t}:=\{x \in \real^{d}~;~f(x) > t\}$.
		%Then $\1_{E_{t}} \in BV(\real^{d})$ almost every $t \in \real$, (see, e.g., Theorem 5.9 (i) in \cite{EvGa92}).
	\end{itemize}
\end{Eg}

\subsubsection*{Orlicz--Sobolev spaces}
We now recall the definition of the Orlicz--Sobolev space $W^{1,\Phi}(\real^{d})$ with a Young function $\Phi$.
For more detail, we refer to \cite{HaHa19,RaRe}.
We first recall the definitions of Young functions and N-functions.
A convex function $\Phi:[0,\infty) \to [0,\infty]$ is called a {\it Young function} if it satisfies the conditions: $\Phi(0)=0$ and $\lim_{x \to \infty}\Phi(x)=\infty$.
A Young function $\Phi$ has the following integral form
\begin{align*}
	\Phi(x)
	=
	\int_{0}^{x}
		\varphi(y)
	\rd y,~x \in [0,\infty),
\end{align*}
where $\varphi:[0,\infty) \to [0,\infty]$ is non-decreasing and left continuous such that $\varphi(0)=0$, and if $\varphi(x)=\infty$ for $x \geq a \geq 0$, then $\Phi(x)=\infty$ for $x \geq a$ (see, e.g.,  Section 1.3, Corollary 2 in \cite{RaRe}).
For a Young function $\Phi$, the {\it complementary function} $\Psi:[0,\infty) \to [0,\infty]$ of $\Phi$ and the {\it generalized inverse} $\Phi^{-1}:[0,\infty] \to [0,\infty]$ of $\Phi$ are defined by
\begin{align*}
	\Psi(x)
	:=
	\sup_{y \geq 0}\{yx-\Phi(y)\}
	=
	\int_{0}^{x}
		\varphi^{-1}(y)
	\rd y
	\quad\text{and}\quad
	\Phi^{-1}(x)
	:=
	\inf\{y \geq 0~;~\Phi(y)>x\},
\end{align*}
where $\varphi^{-1}(x):=\inf\{y \geq 0~;~\varphi(y)>x\}$, $x \geq 0$.
A Young function $\Phi$ is called an {\it N-function} if it satisfies the following conditions: (N-i) $\Phi$ is continuous; (N-ii) $\Phi(x)=0$ if and only if $x=0$; (N-iii) $\lim_{x \to 0}\Phi(x)/x=0$ and $\lim_{x \to \infty}\Phi(x)/x=\infty$.
Here, continuity in the topology of $C([0,\infty); [0,\infty])$ means that $\lim_{y \to x}\Phi(y)=\Phi(x)$ for every point $x \in [0,\infty)$ regardless of whether $\Phi(x)$ is finite or infinite (e.g., page 14 on \cite{HaHa19}).
For a Young function $\Phi$, the Orlicz space $L^{\Phi}(\real^{d})$ is defined by
\begin{align*}
	L^{\Phi}(\real^{d})
	&:=
	\bigcup_{\alpha >0}
	\left\{
		f:\real^{d} \to \real
		~;~
		f\text{ : measurable and }
		\int_{\real^{d}}
			\Phi(\alpha |f(x)|)
		\rd x
		<\infty
	\right\}.
\end{align*}
If $\mathrm{Leb}$-a.e. equal functions are identified, then this function space is a Banach space with the Luxemburg norm
\begin{align*}
	\|f\|_{L^{\Phi}(\real^{d})}
	:=
	\inf
	\left\{
		\lambda>0~;~
		\int_{\real^{d}}
			\Phi
			\left(
				\frac{|f(x)|}{\lambda}
			\right)
			\rd x
			\leq 1
	\right\}
\end{align*}
(e.g., Section 3.3, Theorem 10 in \cite{RaRe}), and if $\Psi$ is the complementary function of $\Phi$, then the generalized H\"older's inequality
\begin{align}\label{eq_GHolder}
	\int_{\real^{d}}
		|f(x)g(x)|
	\rd x
	\leq
	2
	\|f\|_{L^{\Phi}(\real^{d})}
	\|g\|_{L^{\Psi}(\real^{d})}
\end{align}
holds for any $f \in L^{\Phi}(\real^{d})$ and $g \in L^{\Psi}(\real^{d})$ (e.g., Section 3.3, Proposition 1 in \cite{RaRe}).

A Young function $\Phi$ satisfies the $\Delta_{2}$-condition (or is {\it doubling}) if there exists $C>0$ such that for each $x>0$, $\Phi(2x) \leq C \Phi(x)$.
More specifically, various characterizations (equivalent conditions) of the $\Delta_{2}$-condition are described in Section 2.3 of \cite{RaRe}.

For a Young function $\Phi$, the Orlicz--Sobolev space $W^{1,\Phi}(\real^{d})$ is defined by
\begin{align*}
	W^{1,\Phi}(\real^{d})
	:=
	\left\{
		f \in L^{\Phi}(\real^{d})
		~;~
		|Df| \in L^{\Phi}(\real^{d})
	\right\},
\end{align*}
where $Df:=(D_{1}f,\ldots,D_{d}f)^{\top}$ is the vector of the first order weak partial derivatives $D_{i}f$ of $f$ for $i=1,\ldots,d$ (e.g., Section 9.3, Definition 1 in \cite{RaRe}).

\begin{Rem}\label{Rem_Orlicz_1}
\begin{itemize}
\item[(i)] The complementary function $\Psi$ to a Young function $\Phi$ is also a Young function (e.g., page 10 on \cite{RaRe}, or page 14 on \cite{HaHa19} and Lemma 2.4.2 in \cite{HaHa19}).

\item[(ii)]
The complementary function $\Psi$ to an N-function $\Phi$ satisfies (N-ii) and (N-iii).
Indeed, if $\Psi(x)=0$, then it holds that $x \leq \Phi(y)/y$ for $y>0$.
Thus we obtain $x=0$ by $\lim_{y \to 0}\Phi(y)/y=0$.
On the other hand, since $\varphi$ is non-decreasing, it holds that $\Phi(x)/x \leq \varphi(x) \leq \Phi(2x)/x$ for $x>0$.
Hence $\Phi$ satisfies $\lim_{x \to 0}\Phi(x)/x=0$ and $\lim_{x \to \infty}\Phi(x)/x=\infty$ if and only if $\varphi$ satisfies $\lim_{x \to 0}\varphi(x)=0$ and $\lim_{x \to \infty}\varphi(x)=\infty$.
This implies that $\lim_{x \to 0}\Psi(x)/x=0$ and  $\lim_{x \to \infty}\Psi(x)/x=\infty$.
Furthermore, $\Psi$ is left continuous (see, e.g., page 14 on \cite{HaHa19} and Lemma 2.4.2 in \cite{HaHa19}).
In particular, if $\Psi$ is real-valued, then it is an N-function.

\item[(iii)] In general, the complementary function $\Psi$ to a Young function $\Phi$ does not always satisfy the $\Delta_{2}$-condition even if $\Phi$ satisfies the $\Delta_{2}$-condition (e.g., $\Phi(x)=\int_{0}^{x}\log(1+y){\rm d}y=(1+x)\log(1+x)-x$, $x \in [0,\infty)$).

\item[(iv)]
Let $\Phi$ be a Young function.
Then the following inclusion relations hold $W^{1,\Phi}({\mathbb R}^{d}) \subset W_{{\rm loc}}^{1,1}({\mathbb R}^{d}) \subset BV_{{\rm loc}}({\mathbb R}^{d})$.
Indeed, the first relation is shown as follows:
Let $f \in W^{1,\Phi}({\mathbb R}^{d})$ and $V$ be a compact subset of ${\mathbb R}^{d}$.
By the definition of $L^{\Phi}({\mathbb R}^{d})$, there exists $\alpha>0$ such that $\int_{{\mathbb R}^{d}}\Phi(\alpha|Df(x)|){\rm d}x<\infty$.
Note that if $\Phi^{-1}(x)=\infty$, then for any $y \geq 0$, $\Phi(y) \leq x$, and thus by taking $y \to \infty$, we conclude $x=\infty$.
Therefore, we have $\Phi^{-1}(\dashint_{V}\Phi(\alpha|Df(x)|){\rm d}x)<\infty$.
By using Jensen's inequality, we obtain
\begin{align*}
	\int_{V}|Df(x)|{\rm d}x
	&=
	\frac{{\rm Leb}(V)}{\alpha}
	\dashint_{V}
		\alpha|Df(x)|
	{\rm d}x
	\leq
	\frac{{\rm Leb}(V)}{\alpha}
	\Phi^{-1}
	\left(
		\Phi
		\left(
			\dashint_{V}\alpha|Df(x)|{\rm d}x
		\right)
	\right) \\
	&\leq
	\frac{{\rm Leb}(V)}{\alpha}
	\Phi^{-1}
	\left(
		\dashint_{V}\Phi(\alpha|Df(x)|){\rm d}x
	\right)
	<\infty,
\end{align*}
which implies $f \in W^{1,1}_{\mathrm{loc}}(\real^{d})$.

\item[(v)]
If a Young function $\Phi$ satisfies the $\Delta_{2}$-condition, then the Orlicz space $L^{\Phi}({\mathbb R}^{d})$ coincides with the set of all functions $f$ which satisfy $\int_{{\mathbb R}^{d}}\Phi(|f(x)|){\rm d}x<\infty$.

\end{itemize}
\end{Rem}

\begin{Eg}\label{Ex_Orlicz_0}
	\begin{itemize}
		\item[(i)]
		Let $p \in (1,\infty)$ and $\Phi(x):=x^{p}/p$, $x \in [0,\infty)$.
		Then $\Phi$ is an N-functions and satisfies the $\Delta_{2}$-condition.
		Moreover, the Orlicz space $L^{\Phi}({\mathbb R}^{d})$ coincides with the classical Lebesgue space $L^{p}({\mathbb R}^{d})$.
		
		\item[(ii)]
		Let $p>1$ and $\alpha>0$, or $p>1-\alpha$ and $-1\leq \alpha<0$.
		Then the function $\Phi(x):=x^{p} (\log(e+x))^{\alpha}$, $x \in [0,\infty)$ is an N-function.
		Moreover, $\Phi$ and its complementary function $\Psi$ satisfy the $\Delta_{2}$-condition.
		Note that the Orlicz--Sobolev spaces with such $\Phi$ are used and studied in \cite{AdHu03,IwKoOn01}.
		We only check that $\Phi$ and $\Psi$ satisfy the $\Delta_{2}$-condition.
		Let $x>0$.
		Since $\log(e+x) \leq \log(e+2x) \leq 2\log(e+x)$, we obtain $\Phi(2x) \leq 2^{p}\max\{1,2^{\alpha}\}\Phi(x)$.
		On the other hand, if $p>1$ and $\alpha>0$, we obtain for any $y \geq 0$,
		\begin{align*}
			y(2x)-\Phi(y)%y^{p}(\log(e+y))^{\alpha}
			&\leq
			2yx
			-
			y^{p}
			\left\{
				\log
				\left(
					e+\frac{y}{2^{\frac{1}{p-1}}}
				\right)
			\right\}^{\alpha}
			=
			2^{\frac{p}{p-1}}
			\left\{
				\frac{y}{2^{\frac{1}{p-1}}}x
				-
				\Phi
				\left(
					\frac{y}{2^{\frac{1}{p-1}}}
				\right)
			\right\}.
		\end{align*}
		Thus $\Psi(2x) \leq 2^{\frac{p}{p-1}}\Psi(x)$.
		If $p>1-\alpha$ and $-1\leq \alpha<0$, since the function $y \mapsto \log(e+y)$ is concave, we obtain for any $y \geq 0$,
		\begin{align*}
			y(2x)
			-
			\Phi(y)
			%&=
			%2yx
			%-
			%2^{\frac{\alpha}{\alpha+p-1}}
			%y^{p}
			%\left(
			%	\frac{\log(e+y)}{2^{\frac{1}{\alpha+p-1}}}
			%\right)^{\alpha} \\
			&\leq
			2yx
			-
			2^{\frac{\alpha}{\alpha+p-1}}
			y^{p}
			\left\{
				\log\left(
					e+\frac{y}{2^{\frac{1}{\alpha+p-1}}}
				\right)
			\right\}^{\alpha}
			=
			2^{\frac{\alpha+p}{\alpha+p-1}}
			\left\{
				\frac{y}{2^{\frac{1}{\alpha+p-1}}}x
				-
				\Phi\left(
					\frac{y}{2^{\frac{1}{\alpha+p-1}}}
				\right)
			\right\}.
		\end{align*}
		Thus $\Psi(2x) \leq 2^{\frac{\alpha+p}{\alpha+p-1}}\Psi(x)$.
		
		\item[(iii)]
		Let $\Phi$ be a Young function which satisfies the condition (N-iii) and let $f \in L^{1}(\real^{d})$ be a continuous and strictly positive function with $\lim_{|x| \to \infty}f(x)=0$, then $f \in L^{\Phi}(\real^{d})$.
		Indeed, for any $\alpha>0$, there exists $K>0$ such that $C_{\alpha}:=\sup_{|x|>K}\frac{\Phi(\alpha f(x))}{\alpha f(x)}<\infty$.
		Thus since $\Phi$ is non-decreasing and $f$ is bounded on $B(0;K)$, we obtain $\int_{{\mathbb R}^{d}}\Phi(\alpha|f(x)|){\rm d}x \leq {\rm Leb}(B(0;K))\Phi(\alpha\sup_{x \in B(0;K)}|f(x)|)+\alpha C_{\alpha}\|f\|_{L^{1}({\mathbb R}^{d})}<\infty$.
	\end{itemize}
\end{Eg}

\subsubsection*{Sobolev spaces with variable exponents}
We next recall the definition of the Sobolev space $W^{1,p(\cdot)}({\mathbb R}^{d})$ with a variable exponent $p$.
This function space is defined as one of generalized Orlicz spaces (also known as Musielak--Orlicz spaces) by the modular $f \mapsto \int_{{\mathbb R}^{d}}|f(x)|^{p(x)}{\rm d}x$.
For more detail, we refer to \cite{DiHaHaRu11}.

A measurable function $p:\real^{d} \to [1,\infty]$ is called a variable exponent on ${\mathbb R}^{d}$, and denoted by $p \in {\mathcal P}({\mathbb R}^{d})$.
We define
\begin{align*}
	p^{-}
	:=
	\underset{x \in \real^{d}}{{\rm ess~inf~}}
	p(x)
	\quad \text{ and } \quad
	p^{+}
	:=
	\underset{x \in \real^{d}}{{\rm ess~sup~}}
	p(x).
\end{align*}
The Lebesgue space $L^{p(\cdot)}({\mathbb R}^{d})$ with a variable exponent $p \in  {\mathcal P}({\mathbb R}^{d})$ is defined by
\begin{align*}
	L^{p(\cdot)}({\mathbb R}^{d})
	&:=\bigcup_{\alpha>0}
	\left\{
		f:\real^{d} \to \real
		~;~
		f\text{ : measurable and }
		\int_{{\mathbb R}^{d}}
			\left|
				\alpha
				f(x)
			\right|^{p(x)}
		{\rm d}x
		<\infty
	\right\}
\end{align*}
If $\mathrm{Leb}$-a.e. equal functions are identified, then this function space is a Banach space with respect to the norm
\begin{align*}
	\|f\|_{L^{p(\cdot)}({\mathbb R}^{d})}
	:=
	\inf
	\left\{
		\lambda>0
		~;~
		\int_{{\mathbb R}^{d}}
			\left|
				\frac{f(x)}{\lambda}
			\right|^{p(x)}
		{\rm d}x
		\leq 1
	\right\}
\end{align*}
(e.g., Theorem 3.2.7 in \cite{DiHaHaRu11}).
Let $p,q,s \in {\mathcal P}({\mathbb R}^{d})$ and assume that
\begin{align*}
	\frac{1}{s(x)}=\frac{1}{p(x)}+\frac{1}{q(x)},~
	\text{$\mathrm{Leb}$-a.e}.~x \in \real^{d}.
\end{align*}
Then for any $f \in L^{p(\cdot)}({\mathbb R}^{d})$ and $g \in L^{q(\cdot)}({\mathbb R}^{d})$, the generalized H\"older's inequality
\begin{align}\label{eq_GHolder_1}
	\|fg\|_{L^{s(\cdot)}({\mathbb R}^{d})}
	\leq
	2
	\|f\|_{L^{p(\cdot)}({\mathbb R}^{d})}
	\|g\|_{L^{q(\cdot)}({\mathbb R}^{d})}
\end{align}
holds (see, Lemma 3.2.20 in \cite{DiHaHaRu11}).
In the case $s=p=q=\infty$, we use the convention $1/\infty=0$.

We say a function $f: {\mathbb R}^{d} \to {\mathbb R}$ is locally log-H\"older continuous on ${\mathbb R}^{d}$ if there exists $C>0$ such that for any $x, y \in {\mathbb R}^{d}$,
\begin{align*}
	\left|
		f(x)
		-
		f(y)
	\right|
	\leq
	\frac{C}{\log(e+1/|x-y|)}.
\end{align*}
More specifically, various characterizations (equivalent conditions) of the locally log-H\"older continuity are described in Lemma 4.1.6 of \cite{DiHaHaRu11}.
We say that $f: {\mathbb R}^{d} \to {\mathbb R}$ satisfies the log-H\"older decay condition if there exist $f_{\infty} \in {\mathbb R}$ and $C>0$ such that for any $x \in {\mathbb R}^{d}$,
\begin{align}
\label{eq:7_1}
	\left|
		f(x)
		-
		f_{\infty}
	\right|
	\leq
	\frac{C}{\log(e+|x|)}.
\end{align}
We say that $f: {\mathbb R}^{d} \to {\mathbb R}$ is globally log-H\"older continuous on ${\mathbb R}^{d}$ if it is locally log-H\"older continuous on ${\mathbb R}^{d}$ and satisfies the log-H\"older decay condition.
If $f: {\mathbb R}^{d} \to {\mathbb R}$ is globally log-H\"older continuous on ${\mathbb R}^{d}$, then the constant $f_{\infty}$ in \eqref{eq:7_1} is unique and $f$ is bounded (e.g., page 100 on \cite{DiHaHaRu11}).
We define
\begin{align*}
	{\mathcal P}^{\log}({\mathbb R}^{d})
	:=
	\left\{
		p \in {\mathcal P}({\mathbb R}^{d})
		~;~
		1/p \text{ is globally log-H\"older continuous}
	\right\}
\end{align*}
and define $p_{\infty}$ by $1/p_{\infty}:=\lim_{|x| \to \infty}1/p(x)$.
As usual we use the convention $1/\infty:=0$.

The Sobolev space $W^{1,p(\cdot)}({\mathbb R}^{d})$ with a variable exponent $p \in {\mathcal P}({\mathbb R}^{d})$ is defined by
\begin{align*}
	W^{1,p(\cdot)}({\mathbb R}^{d})
	:=
	\left\{
		f \in L^{p(\cdot)}({\mathbb R}^{d})
		~;~
		|Df|
		\in
		L^{p(\cdot)}({\mathbb R}^{d})
	\right\},
\end{align*}
where $Df:=(D_{1}f,\ldots,D_{d}f)^{\top}$ is the vector of the first order weak partial derivatives $D_{i}f$ of $f$ for $i=1,\ldots,d$ (see, e.g., Definition 8.1.2 in \cite{DiHaHaRu11}).

\begin{Rem}\label{Rem_Sob_expo}
\begin{itemize}
	\item[(i)]
	For $p \in {\mathcal P}^{\log}({\mathbb R}^{d})$, although $1/p$ is bounded, $p$ is not always bounded (e.g., page 101 on \cite{DiHaHaRu11}).
	
	\item[(ii)]
	$p \in {\mathcal P}^{\log}({\mathbb R}^{d})$ if and only if $p^{*}:=p/(p-1) \in {\mathcal P}^{\log}({\mathbb R}^{d})$, and then $(p_{\infty})^{*}=(p^{*})_{\infty}$ (e.g., page 101 on \cite{DiHaHaRu11}).
	
	\item[(iii)]
	For $p \in {\mathcal P}({\mathbb R}^{d})$ with $p^{+}<\infty$, $p \in {\mathcal P}^{\log}({\mathbb R}^{d})$ if and only if $p$ is globally log-H\"older continuous.
	This is due to the fact that $p \mapsto 1/p$ is a bilipschitz mapping from $[p^{-},p^{+}]$ to $[1/p^{+},1/p^{-}]$ (e.g., Remark 4.1.5 in \cite{DiHaHaRu11}).

	\item[(iv)]
	Note that the following inclusion relations hold $W^{1,p(\cdot)}({\mathbb R}^{d}) \subset W_{{\rm loc}}^{1,p^{-}}({\mathbb R}^{d}) \subset W_{{\rm loc}}^{1,1}({\mathbb R}^{d})$.
	Indeed, the first relation is shown by using the generalized H\"older's inequality \eqref{eq_GHolder_1}.
\end{itemize}
\end{Rem}

\begin{Eg}
Let $d=1$ and
\begin{align*}
	p(x)
	:=
	\max
	\left\{
		1-e^{3-|x|},
		\min
		\left\{
			\frac{6}{5},
			\max
			\left\{
				\frac{1}{2},
				\frac{3}{2}-x^{2}
			\right\}
		\right\}
	\right\}
	+1,
	\quad
	x \in {\mathbb R}.
\end{align*}
Then $p \in {\mathcal P}^{\log}({\mathbb R})$ and $1<p^{-}<p^{+}<\infty$ (e.g., Example 1.3 in \cite{NaSa12}, or Example 9.1.15 and Example 9.1.16 in \cite{YaLiKy17}).
\end{Eg}

\subsubsection*{Fractional Sobolev spaces}
We finally recall the definition of the fractional Sobolev space $W^{s,p}({\mathbb R}^{d})$.
For more detail, we refer to \cite{DiPaVa12}.

Let $s \in (0,1)$ be a fractional exponent and $p \in [1,\infty)$.
The fractional Sobolev space $W^{s,p}({\mathbb R}^{d})$ is defined by
\begin{align*}
	W^{s,p}({\mathbb R}^{d})
	:=
	\left\{
		f \in L^{p}({\mathbb R}^{d})
		~;~
		\int_{{\mathbb R}^{d}}
			\int_{{\mathbb R}^{d}}
				\left(
					\frac{|f(x)-f(y)|}{|x-y|^{d/p+s}}
				\right)^{p}
			\rd x
		\rd y
		<\infty
	\right\}.
\end{align*}
In the literature, the fractional Sobolev space is also called the Aronszajn, Gagliardo or Slobodeckij space.

For $s \in (0,1)$ and $p \in [1,\infty)$ and $f \in W^{s,p}(\real^{d})$, we define the operator $G_{s,p}$ by
\begin{align*}
	G_{s,p}f(x)
	:=
	\left(
		\int_{{\mathbb R}^{d}}
		\left(
			\frac{|f(x)-f(y)|}{|x-y|^{d/p+s}}
		\right)^{p}
		{\rm d}y
	\right)^{1/p},~
	x \in {\mathbb R}^{d}.
\end{align*}
Then $G_{s,p}f \in L^{p}(\real^{d})$.

\subsection{Hardy--Littlewood maximal function and estimates}\label{sec_2_2}
In this subsection, we recall the definition of the Hardy--Littlewood maximal function and their estimates on function spaces defined in Section \ref{sec_2_1}, which are well-known in the fields of real analysis and harmonic analysis.

Let $\nu$ be a locally finite vector valued measure on $\real^{d}$.
The Hardy--Littlewood maximal operator $M$ for $\nu$ is defined by
\begin{align*}
	M\nu(x)
	:=
	\sup_{s>0}
	\dashint_{B(x;s)}
	\rd |\nu|(z),
	\quad
	\dashint_{B(x;s)}
	\rd |\nu|(z)
	:=
	\frac{|\nu|(B(x;s))}{\mathrm{Leb}(B(x;s))},
\end{align*}
where $|\nu|$ is the total variation of $\nu$.
For $R>0$,  we define the restricted Hardy--Littlewood maximal function $M_{R}\nu$ by
\begin{align*}
	M_{R}\nu(x)
	:=
	\sup_{0<s\leq R}
	\dashint_{B(x;s)}
	\rd |\nu|(z),~
	x \in {\mathbb R}^{d}.
\end{align*}
If $\nu(\rd x)=f(x) \rd x$, then we write $Mf(x)$ and $M_{R}f(x)$, respectively.

The following lemma is well-known as the Hardy--Littlewood maximal  weak and strong type estimates.

\begin{Lem}\label{Lem_key_2}
	\begin{itemize}
	\item[(i)]
	Weak type estimate (e.g., Chapter III, Section 4.1, (a) in \cite{St70}).
	There exists $A_{1}>0$ such that for any finite and vector valued signed measure measure $\nu$ on $\real^{d}$ and $\lambda>0$,
	\begin{align*}
		\mathrm{Leb}
		\left(\left\{x\in \real^{d}~;~M\nu(x) > \lambda\right\}\right)
		\leq
		A_{1}
		|\nu|(\real^{d})
		\lambda^{-1}.
	\end{align*}
	\item[(ii)]
	Strong type estimate (e.g., Chapter I, Section 1.3, Theorem 1 (c) in \cite{St70}).
	For any $p \in (1,\infty]$, there exists $A_{p}>0$ such that for any $f \in L^{p}(\real^{d})$,
	\begin{align*}
		\|Mf\|_{L^{p}(\real^{d})}
		\leq
		A_{p} \|f\|_{L^{p}(\real^{d})}.
	\end{align*}
	\end{itemize}
\end{Lem}

\begin{Rem}
	\begin{itemize}
		\item[(i)]
		The estimate in Lemma \ref{Lem_key_2} (i) can be shown by the same way as the proof of Theorem 1 (b) in Chapter I, Section 1.3 of \cite{St70} as an application of Vitali's covering lemma (e.g., Chapter I, Section 1.6, Lemma in \cite{St70}), and the constant $A_{1}$ can be chosen as $A_{1}=5^{d}$.
		
		\item[(ii)]
		The Hardy--Littlewood maximal operator is used to prove the flow property of ordinary differential equations (ODEs) and stochastic differential equations (SDEs) with Sobolev coefficients.
		In particular, by using this maximal operator, Crippa and De Lellis \cite{CrLe08} proved the existence of a unique regular Lagrangian flow for ODEs with a local Sobolev coefficient, and Zhang \cite{Zh11} (also see, \cite{Zh13}) studied the stochastic homeomorphism flows property for SDEs with local Sobolev coefficients.
	\end{itemize}
\end{Rem}

The following lemma shows that the $\Delta_{2}$-condition is equivalent to the Hardy--Littlewood maximal strong type estimate on the Orlicz space.

\begin{Lem}[Theorem 2.1 in \cite{Ga88}]\label{Lem_key_3}
	Let $\Phi$ be an N-function and $\Psi$ be its complementary function.
	Then $\Psi$ satisfies the $\Delta_{2}$-condition if and only if there exists $A_{\Phi}>0$ such that for any $f \in L^{\Phi}(\real^{d})$,
	\begin{align*}
		\|Mf\|_{L^{\Phi}({\mathbb R}^{d})}
		\leq
		A_{\Phi}
		\|f\|_{L^{\Phi}({\mathbb R}^{d})}.
	\end{align*}
\end{Lem}

The following lemma shows that the Hardy--Littlewood maximal strong type estimate holds on the Sobolev space $W^{1,p(\cdot)}(\real^{d})$ with a variable exponents $p \in {\mathcal P}^{\log}({\mathbb R}^{d})$.

\begin{Lem}[e.g., Theorem 4.3.8 in \cite{DiHaHaRu11}]
	\label{lem:0.3}
	Let $p \in {\mathcal P}^{\log}({\mathbb R}^{d})$ with $1<p^{-}$.
	Then there exists $A_{p(\cdot)}>0$ such that for any $f \in L^{p(\cdot)}({\mathbb R}^{d})$,
	\begin{align*}
		\|Mf\|_{L^{p(\cdot)}({\mathbb R}^{d})}
		\leq
		A_{p(\cdot)}
		\|f\|_{L^{p(\cdot)}({\mathbb R}^{d})}.
	\end{align*}
\end{Lem}

\subsection{Main results}\label{sec_2_3}
In this subsection, we state our main results of this article as multi-dimensional versions of Avikainen's estimate.
We use notations $1/\infty:=0$ and $1/0:=\infty$ for convenience.

We first consider the case of bounded variation in $\real^{d}$.

\begin{Thm}\label{main_0}
Let $X, \widehat{X}:\Omega \to \real^{d}$ be random variables which admit density functions $p_{X}$ and $p_{\widehat{X}}$ with respect to Lebesgue measure, respectively, and let $r \in (1,\infty]$.
Suppose that $p_{X},p_{\widehat{X}} \in L^{\infty}(\real^{d})$.
Then for any $f \in BV(\real^{d}) \cap L^{r}(\real^{d},p_{X}) \cap L^{r}(\real^{d},p_{\widehat{X}})$, $p \in (0,\infty)$ and $q \in [1,r)$, it holds that
	\begin{align}\label{main_0_1}
		\e\left[
			\left|
				f(X)
				-
				f(\widehat{X})
			\right|^{q}
		\right]
		\leq
		C_{BV}(p,q,r)
\e\left[\left|X-\widehat{X}\right|^{p}\right]^{\frac{1-q/r}{p+1}},
\end{align}
where the constant $C_{BV}(p,q,r)$ is defined by
\begin{align*}
	&C_{BV}(p,q,r)
	\\&:=
	\left\{
	\begin{array}{lll}
		\begin{array}{l}
		\displaystyle
			(2\|f\|_{\infty})^{q-1}
			\Big(
				2^{p+1}K_{0}^{p}
				\|f\|_{\infty}
				\\
				\hspace{1.92cm}
				\displaystyle
				+
				A_{1}
				\{
					\|p_{X}\|_{\infty}
					+
					\|p_{\widehat{X}}\|_{\infty}
				\}
				\int_{\real^{d}}
					|Df|
			\Big),
		\end{array}
		&\text{if}\quad
		r=\infty,\\
		\begin{array}{l}
			\hspace{-0.17cm}
			\displaystyle{
				2^{q-1}
				\left(
					2^{p+1}K_{0}^{p}
					+
					A_{1}
					\{
						\|p_{X}\|_{\infty}
						+
						\|p_{\widehat{X}}\|_{\infty}
					\}
					\int_{{\mathbb R}^{d}}
						|Df|
				\right)
			} \\
			\hspace{1.92cm}
			\displaystyle
			+
			\left(
				\|f\|_{L^{r}(\real^{d},p_{X})}
				+
				\|f\|_{L^{r}(\real^{d},p_{\widehat{X}})}
			\right)^{r},
		\end{array}
		&\text{if}\quad
		r \in (1,\infty),
		& {\mathbb E}[|X-\widehat{X}|^{p}]<1, \\
			\displaystyle
				\left(
					\|f\|_{L^{r}(\real^{d},p_{X})}
					+
					\|f\|_{L^{r}(\real^{d},p_{\widehat{X}})}
				\right)^{q},
		&\text{if}\quad
		r \in (1,\infty),
		& {\mathbb E}[|X-\widehat{X}|^{p}] \geq 1.
	\end{array}\right.
\end{align*}
Here, $K_{0}$ and $A_{1}$ are the constants of the pointwise estimate \eqref{Lem_key_0_1} in Lemma \ref{Lem_key_0} and of the Hardy--Littlewood maximal  weak type estimate in Lemma \ref{Lem_key_2} (i), respectively.
\end{Thm}

\begin{Rem}
\label{rem:2.12}
	\begin{itemize}
		\item[(i)]
		In Theorem \ref{main_0}, we need the existence of bounded density functions for both $X$ and $\widehat{X}$ in order to use the poinwise estimate \eqref{Lem_key_0_1} in Lemma \ref{Lem_key_0}.
		%However, we assume the boundedness only for one of them.
		
		\item[(ii)]
		For the one-dimensional case, Theorem \ref{main_0} with $r=\infty$ requires the stronger assumption (i.e., the existence of bounded density functions for both $X$ and $\widehat{X}$) than the result of Avikainen \cite{Av09}.
		Since the constant $C_{BV}(p,q,\infty)$ depends on $\|f\|_{\infty}$ and $K_{0}$, it might be difficult to compare the constants in the upper bound of \eqref{Av_0} and \eqref{main_0_1} in general.
		
		\item[(iii)]
		By using the same way as the proof of Theorem 2.4 (ii) in \cite{Av09}, we can prove that in the estimate \eqref{main_0_1} for $r=\infty$, the power $1/(p+1)$ is optimal, that is, there exist $f \in BV(\real^{d}) \cap L^{\infty}({\mathbb R}^{d})$ and random variables $X$ and $\widehat{X}$ with the bounded density functions such that both sides in \eqref{main_0_1} coincide for some constant $C_{BV}(p,q,\infty)$.
		
		\item[(iv)]
		In the case of $r \in (1,\infty)$ and ${\mathbb E}[|X-\widehat{X}|^{p}] \geq 1$, the power of the right hand side of the estimate \eqref{main_0_1} does not necessarily have to be $\frac{1-q/r}{p+1}$ and can be chosen arbitrarily.
		Indeed, for any $\alpha \geq 0$, since ${\mathbb E}[|X-\widehat{X}|^{p}]^{\alpha} \geq 1$, we obtain
		\begin{align*}
		{\mathbb E}\left[\left|f(X)-f(\widehat{X})\right|^{q}\right]
		&\leq \left(\|f\|_{L^{r}(\real^{d},p_{X})}+\|f\|_{L^{r}(\real^{d},p_{\widehat{X}})}\right)^{q} \\
		&\leq \left(\|f\|_{L^{r}(\real^{d},p_{X})}+\|f\|_{L^{r}(\real^{d},p_{\widehat{X}})}\right)^{q}{\mathbb E}\left[\left|X-\widehat{X}\right|^{p}\right]^{\alpha}.
		\end{align*}
		
		\item[(v)]
		If $r \in (1,\infty)$, then the constant $C_{BV}(p,q,r)$ depends on both densities $p_{X}$ and $p_{\widehat{X}}$ or expectations $\e[|f(X)|^{r}]$ and $\e[|f(\widehat{X})|^{r}]$.
		However, in the application to the multilevel Monte Carlo method in Section \ref{sec_3}, it might be possible to improve these dependence by using the Gaussian upper bound (see, \eqref{GB_1}, \eqref{GB_2} and Remark \ref{Rem_GB_0} (ii)).
	\end{itemize}
\end{Rem}

Before proving Theorem \ref{main_0}, we give a pointwise estimate for functions of locally bounded variation in $\real^{d}$, which plays a crucial role in our arguments.

\begin{Lem}
	\label{Lem_key_0}
	Let $f \in BV_{\mathrm{loc}}(\real^{d})$.
	Then there exist a constant $K_{0}>0$ and a Lebesgue null set $N \in {\mathscr B}({\mathbb R}^{d})$ such that for all $x,y \in \real^{d} \setminus N$,
	\begin{align}
		|f(x)-f(y)|
		&\leq
		K_{0}
		|x-y|
		\left\{
			M_{2|x-y|}(Df)(x)
			+
			M_{2|x-y|}(Df)(y)
		\right\},
		\label{Lem_key_0_1}
	\end{align}
\end{Lem}

\begin{Rem}\label{Rem_0}
	%As a consequence of Lemma \ref{Lem_key_0}, the following pointwise estimate holds:
	%\begin{align}\label{pw_esti_1}
	%	|f(x)-f(y)|
	%	&\leq
	%	K_{0}
	%	|x-y|
	%	\left\{
	%		M_{2|x-y|}(Df)(x)
	%		+
	%		M_{2|x-y|}(Df)(y)
	%	\right\},
	%	\text{ $\mathrm{Leb}$-a.e. }x,y \in \real^{d}.
	%\end{align}
	\begin{itemize}
		\item[(i)]
		Note that Theorem 3 in \cite{LaTu14} shows that functions in $BV(\real^{d})$ can be characterized by the estimate \eqref{Lem_key_0_1}.
		
		\item[(ii)]
		In Theorem \ref{main_0}, we need to assume boundedness for the density functions of both random variables $X$ and $\widehat{X}$, which is a stronger assumption than in the one-dimensional case (see, Theorem 2.4(i) in \cite{Av09}).
		Here, if a ``nonsymmetric" version $|f(x)-f(y)|\leq K_{0}|x-y|M_{2|x-y|}(Df)(x)$, $\mathrm{Leb}$-a.e. $x, y \in \real^{d}$ of the pointwise estimate is correct, then we can remove the boundedness of the density function of either $X$ or $\widehat{X}$ in Theorem \ref{main_0}.
		%For example, if we can prove a ``nonsymmetric" version of a pointwise estimate $|f(x)-f(y)|\leq K_{0}|x-y|M_{2|x-y|}(Df)(x)$, $\mathrm{Leb}$-a.e. $x, y \in \real^{d}$, then by assuming the boundedness of the density function in either $X$ or $\widehat{X}$, Avikainen's estimates in Theorem \ref{main_0} holds.
	\end{itemize}

	%If both $X$ and $\widehat{X}$ admit bounded density functions, we can use the pointwise estimate \eqref{pw_esti_1} for proving Theorem \ref{main_0} and Lemma \ref{Lem_key_1}.
	%, but we assume boundedness only for one of them.
	%Thus we need to modify the pointwise estimate from \eqref{pw_esti_1} to \eqref{Lem_key_0_1}.
\end{Rem}

The estimate \eqref{Lem_key_0_1} is basically known in the field of harmonic analysis.
%However, the authors could not find proper references.
For the convenience of readers, we will give a proof below.

\begin{proof}[Proof of Lemma \ref{Lem_key_0}]
	The proof is based on Theorem 3.2 in \cite{HaKo00}.
	We first note that if $d \geq 2$, by using Jensen's inequality and Poincar\'e's inequality for functions of locally bounded variation (see, e.g., Theorem 5.10 (ii) in \cite{EvGa92}), there exists a constant $C_{0}>0$ such that for any $x \in \real^{d}$ and $r>0$,
	\begin{align}\label{Lem_key_0_2}
		\dashint_{B(x;r)}
			\left|
				f(z)
				-
				(f)_{x,r}
			\right|
		\rd z
		&\leq
		\left(
			\dashint_{B(x;r)}
				\left|
					f(z)
					-
					(f)_{x,r}
				\right|^{\frac{d}{d-1}}
			\rd z
		\right)^{\frac{d-1}{d}} \notag\\
		&\leq
		\frac{C_{0}}{\mathrm{Leb}(U(x;r))^{\frac{d-1}{d}}}
		\int_{U(x;r)}|Df| \notag\\
		&\leq
		C_{d}
		r
		M_{r}(Df)(x),
	\end{align}
	where $(f)_{x,r}:=\dashint_{B(x;r)} f(z) \rd z$ and $C_{d}:=C_{0} \sqrt{\pi} \Gamma(d/2+1)^{-1/d}$.
	If $d=1$, there exists $\{f_{k}\}_{k \in \n} \subset C^{1}(U(x;r); {\mathbb R})$ such that $f_{k} \to f$ in $L^{1}(U(x;r))$ and $\int_{U(x;r)} |f_{k}'(z)| \rd z \to \int_{U(x;r)}|Df|$ as $k \to \infty$ (e.g., Theorem 5.3 in \cite{EvGa92} or Theorem 1.17 in \cite{Gi84}).
	Then by using Fatou's lemma and Lemma 4.1 in \cite{EvGa92} with $p=1$, there exists $C_{1}>0$ such that
	\begin{align}\label{Lem_key_0_20}
		\dashint_{B(x;r)}
			\left|
				f(z)
				-
				(f)_{x,r}
			\right|
		\rd z
		&\leq
		\liminf_{k \to \infty}
		\dashint_{U(x;r)}
			\dashint_{U(x;r)}
			\left|
				f_{k}(z)
				-
				f_{k}(y)
			\right|
			\rd y
		\rd z \notag\\
		&\leq C_{1}r
		\liminf_{k \to \infty}
		\dashint_{U(x;r)}
			\dashint_{U(x;r)}
				|f_{k}'(y)|
			\rd y
		\rd z
		\notag
		\\
		&=C_{1}r
		\dashint_{U(x;r)} |Df|
		\leq
		C_{1}rM_{r}(Df)(x).
	\end{align}
	Moreover, the Lebesgue differentiation theorem (e.g., Theorem 1.32 in \cite{EvGa92}) shows that there exists a Lebesgue null set $N \in {\mathscr B}({\mathbb R}^{d})$ such that for any $x \in \real^{d} \setminus N$,
	\begin{align}\label{Lem_key_0_3}
		\lim_{r \to 0}
		(f)_{x,r}
		=
		f(x).
	\end{align}
	Let $x, y \in \real^{d} \setminus N$ be fixed and set $r_{i}:=2^{-i}|x-y|$ for $i \in \n\cup\{0\}$.
	Then by using \eqref{Lem_key_0_3}, we obtain
	\begin{align*}
		|f(x)-(f)_{x,r_{0}}|
		&\leq
		\sum_{i=0}^{\infty}
			\left|
				(f)_{x,r_{i+1}}
				-
				(f)_{x,r_{i}}
			\right|\notag\\
		&\leq
		\sum_{i=0}^{\infty}
		\dashint_{B(x;r_{i+1})}
		\left|
			f(z)
			-
			(f)_{x,r_{i}}
		\right|
		\rd z\notag\\
		&\leq
		2^{d}
		\sum_{i=0}^{\infty}
		\dashint_{B(x;r_{i})}
			\left|
			f(z)
			-
			(f)_{x,r_{i}}
		\right|
		\rd z.
	\end{align*}
	Therefore, it follows from \eqref{Lem_key_0_2} or \eqref{Lem_key_0_20} that
	\begin{align}\label{Lem_key_0_4}
		|f(x)-(f)_{x,r_{0}}|
		&\leq
		2^{d+1}
		C_{d}
		|x-y|
		M_{|x-y|}(Df)(x).
	\end{align}
	By the same way,
	we have
	\begin{align}\label{Lem_key_0_5}
		|f(y)-(f)_{y,r_{0}}|
		\leq
		2^{d+1}
		C_{d}
		|x-y|
		M_{|x-y|}(Df)(y).
	\end{align}
	On the other hand, it holds from \eqref{Lem_key_0_2} or \eqref{Lem_key_0_20} that
	\begin{align}\label{Lem_key_0_6}
		|(f)_{x,r_{0}}-(f)_{y,r_{0}}|
		&\leq
		|(f)_{x,r_{0}}-(f)_{x,2r_{0}}|
		+
		|(f)_{x,2r_{0}}-(f)_{y,r_{0}}| \notag\\
		&\leq
		\dashint_{B(x;r_{0})}
			|f(z)-(f)_{x,2r_{0}}|
		\rd z
		+
		\dashint_{B(y;r_{0})}
			|(f)_{x,2r_{0}}-f(z)|
		\rd z \notag\\
		&\leq
		2^{d+1}
		\dashint_{B(x;2r_{0})}
			|f(z)-(f)_{x,2r_{0}}|
		\rd z \notag\\
		&\leq
		2^{d+2}C_{d}|x-y|M_{2|x-y|}(Df)(x).
	\end{align}
	By combining \eqref{Lem_key_0_4}, \eqref{Lem_key_0_5} and \eqref{Lem_key_0_6}, we conclude the proof.
\end{proof}

By using the Hardy--Littlewood maximal weak type estimate in Lemma \ref{Lem_key_2} (i) and the pointwise estimate \eqref{Lem_key_0_1} in Lemma \ref{Lem_key_0}, we first prove the estimate \eqref{main_0_1} for indicator functions $\1_{E} \in BV(\real^{d})$, which is a multi-dimensional version of Lemma 3.4 in \cite{Av09} and Proposition 5.3 in \cite{GiXi17}.

\begin{Lem}\label{Lem_key_1}
	Let $X, \widehat{X}:\Omega \to \real^{d}$ be random variables which admit density functions $p_{X}$ and $p_{\widehat{X}}$ with respect to Lebesgue measure, respectively.
	Suppose that $p_{X}, p_{\widehat{X}} \in L^{\infty}(\real^{d})$.
	Then for any $E \in \mathscr{B}(\real^{d})$ with ${\bf 1}_{E} \in BV(\real^{d})$ and $p,q \in (0,\infty)$, it holds that
	\begin{align}\label{main_0_2}
		\e\left[
			\left|
				\1_{E}(X)
				-
				\1_{E}(\widehat{X})
			\right|^{q}
		\right]
		\leq
		\left(
			(2K_{0})^{p}
			+
			A_{1}
			\{
				\|p_{X}\|_{\infty}
				+
				\|p_{\widehat{X}}\|_{\infty}
			\}
			\int_{\real^{d}}
				|D\1_{E}|
		\right)
		\e\left[
			\left|
				X
				-
				\widehat{X}
			\right|^{p}
		\right]^{\frac{1}{p+1}}.
	\end{align}
\end{Lem}
\begin{proof}
	If $\e[|X-\widehat{X}|^{p}]=0$ then $X=\widehat{X}$ almost surely, and thus the statement is obvious.
	
	We assume $\e[|X-\widehat{X}|^{p}]>0$.
	For $\lambda>0$, we define the event $\Omega(D \1_{E},\lambda) \in \mathscr{F}$ by
	\begin{align*}
		\Omega(D \1_{E},\lambda)
		:=
		\left\{
			M(D \1_{E})(X)>\lambda
		\right\}
		\cup
		\left\{
			M(D \1_{E})(\widehat{X})>\lambda
		\right\}.
	\end{align*}
	We first remark that for any $x,y \in \real^{d}$, it holds that
	\begin{align}\label{Lem_key_1_1}
		|\1_{E}(x)-\1_{E}(y)|^{q}
		=
		|\1_{E}(x)-\1_{E}(y)|^{p}.
	\end{align}
	By using this trick, we obtain
	\begin{align*}
		{\mathbb E}
		\left[
			\left|
				\1_{E}(X)
				-
				\1_{E}(\widehat{X})
			\right|^{q}
		\right]
		=
		{\mathbb E}
		\left[
			\left|
				\1_{E}(X)
				-
				\1_{E}(\widehat{X})
			\right|^{p}
			{\bf 1}_{\Omega(D \1_{E},\lambda)}
		\right]
		+
		{\mathbb E}
		\left[
			\left|
				\1_{E}(X)
				-
				\1_{E}(\widehat{X})
			\right|^{p}
			{\bf 1}_{\Omega(D \1_{E},\lambda)^{{\rm c}}}
		\right].
	\end{align*}
	On the event $\Omega(D \1_{E},\lambda)$, since $X$ and $\widehat{X}$ have bounded density functions, by using Lemma \ref{Lem_key_2} (i), we have
	\begin{align}
		\label{eq:7}
		{\mathbb E}
			\left[
				\left|
					\1_{E}(X)
					-
					\1_{E}(\widehat{X})
				\right|^{p}
				{\bf 1}_{\Omega(D \1_{E},\lambda)}
			\right]
		&\leq
		{\mathbb P}(M(D{\bf 1}_{E})(X)>\lambda)
		+
		{\mathbb P}(M(D{\bf 1}_{E})(\widehat{X})>\lambda)
		\notag\\&
		\leq
		A_{1}
		\{
			\|p_{X}\|_{\infty}
			+
			\|p_{\widehat{X}}\|_{\infty}
		\}
		\int_{\real^{d}}
			|D\1_{E}|
		\lambda^{-1}.
	\end{align}
	Let $N \in {\mathscr B}({\mathbb R}^{d})$ be the Lebesgue null set provided by Lemma \ref{Lem_key_0} for $f=\1_{E}$.
	On the event $\Omega(D \1_{E},\lambda)^{{\rm c}}$, since $X$ and $\widehat{X}$ have density functions, by Lemma \ref{Lem_key_0}, we obtain
	\begin{align}
	\label{eq:8}
		{\mathbb E}
			\left[
				\left|
					\1_{E}(X)
					-
					\1_{E}(\widehat{X})
				\right|^{p}
				{\bf 1}_{\Omega(D \1_{E},\lambda)^{{\rm c}}}
			\right]
		&=
		{\mathbb E}
		\left[
			\left|
				\1_{E}(X)
				-
				\1_{E}(\widehat{X})
			\right|^{p}
			{\bf 1}_{\Omega(D \1_{E},\lambda)^{{\rm c}}}
			\1_{\real^{d} \setminus N}(X)
			\1_{\real^{d} \setminus N}(\widehat{X})
		\right]
		\notag\\
		&\leq
		K_{0}^{p}
		{\mathbb E}
			\left[
				\left|
					X
					-
					\widehat{X}
				\right|^{p}
				\{
					M(D\1_{E})(X)
					+
					M(D\1_{E})(\widehat{X})
				\}^{p}
				{\bf 1}_{\Omega(D \1_{E},\lambda)^{{\rm c}}}
			\right]
		\notag \\
		&\leq
		(2K_{0})^{p}
		\lambda^{p}
		{\mathbb E}
			\left[
				\left|
					X
					-
					\widehat{X}
				\right|^{p}
			\right].
	\end{align}
	Hence, by \eqref{eq:7} and \eqref{eq:8}, we have
	\begin{align*}
		{\mathbb E}
			\left[
				\left|
					\1_{E}(X)
					-
					\1_{E}(\widehat{X})
				\right|^{q}
			\right]
		\leq
			A_{1}
			\{
				\|p_{X}\|_{\infty}
				+
				\|p_{\widehat{X}}\|_{\infty}
			\}
			\int_{\real^{d}}
			|D\1_{E}|\lambda^{-1}
			+
			(2K_{0})^{p}
			\lambda^{p}
			{\mathbb E}
			\left[
				\left|
					X
					-
					\widehat{X}
				\right|^{p}
			\right].
	\end{align*}
	Now we choose $\lambda:={\mathbb E}[|X-\widehat{X}|^{p}]^{-R}>0$ for some $R>0$.
	Then we obtain
	\begin{align*}
		{\mathbb E}
		\left[
			\left|
				\1_{E}(X)
				-
				\1_{E}(\widehat{X})
			\right|^{q}
		\right]
		\leq
		A_{1}
		\{
			\|p_{X}\|_{\infty}
			+
			\|p_{\widehat{X}}\|_{\infty}
		\}
		\int_{\real^{d}}
			|D\1_{E}|
			{\mathbb E}
			\left[
				\left|
					X
					-
					\widehat{X}
				\right|^{p}
			\right]^{R}
			+
			(2K_{0})^{p}
			{\mathbb E}
			\left[
				\left|
					X
					-
					\widehat{X}
				\right|^{p}
			\right]^{1-pR}.
	\end{align*}
	By choosing $R$ as $R=1-pR$, that is, $R=\frac{1}{p+1}$, then we have
	\begin{align*}
		{\mathbb E}
		\left[
			\left|
				\1_{E}(X)
				-
				\1_{E}(\widehat{X})
			\right|^{q}
		\right]
		\leq
		\left(
			A_{1}
			\{
				\|p_{X}\|_{\infty}
				+
				\|p_{\widehat{X}}\|_{\infty}
			\}
			\int_{\real^{d}}
			|D\1_{E}|
			+
			(2K_{0})^{p}
		\right)
		{\mathbb E}
		\left[
			\left|
				X
				-
				\widehat{X}
			\right|^{p}
		\right]^{\frac{1}{p+1}},
	\end{align*}
	which concludes the statement.
\end{proof}

\begin{Rem}
	Note that the equation \eqref{Lem_key_1_1} is the key trick for replacing the power $q$ in the left hand side of Avikainen's estimates \eqref{main_0_1} and \eqref{main_0_2} by $p$ in the right hand side.
\end{Rem}

By using Lemma \ref{Lem_key_1} with the coarea formula for functions of bounded variation, we now prove Theorem \ref{main_0} for general functions $f \in BV(\real^{d})$.

\begin{proof}[Proof of Theorem \ref{main_0}]
For $\lambda>0$, we define the event $\Omega(f,\lambda) \in \mathscr{F}$ by
	\begin{align*}
		\Omega(f,\lambda)
		:=
		\{
			|f(X)|>\lambda
		\}
		\cup
		\{
			|f(\widehat{X})|>\lambda
		\}.
	\end{align*}
	For $t \in \real$, we define $E_{t}:=\{x \in \real^{d}~;~f(x) > t\}$.
	Then for any $x,y \in \real^{d}$, it holds that
	\begin{align*}
		|f(x)-f(y)|
		=
		\int_{f(x) \wedge f(y)}^{f(x)\vee f(y)}
			\left|
				\1_{E_{t}}(x)
				-
				\1_{E_{t}}(y)
			\right|
		\rd t.
	\end{align*}
	Hence, since $q \in [1,\infty)$, by using Jensen's inequality, it holds that
	\begin{align*}
		\e\left[
			\left|
				f(X)
				-
				f(\widehat{X})
			\right|^{q}
			\1_{\Omega(f,\lambda)^{{\rm c}}}
		\right]
		\leq
		(2\lambda)^{q-1}
		\int_{-\lambda}^{\lambda}
			\e\left[
				\left|
					\1_{E_{t}}(X)
					-
					\1_{E_{t}}(\widehat{X})
				\right|^{q}
			\right]
		\rd t.
	\end{align*}
	It follows from Theorem 5.9 (i) in \cite{EvGa92} that $\1_{E_{t}} \in BV(\real^{d})$ for $\mathrm{Leb}$-a.e. $t \in \real$.
	Note that by the coarea formula for functions of bounded variation (see, e.g., Theorem 5.9 (ii) in \cite{EvGa92}), it holds that
	\begin{align*}
		\int_{\real}
			\rd t
			\int_{\real^{d}}
				|D\1_{E_{t}}|
		=
	\int_{\real^{d}}
		|Df|.
	\end{align*}
	Therefore, by using Lemma \ref{Lem_key_1} with $E=E_{t}$, we obtain
	\begin{align}\label{eq_0_2}
		&\e\left[
			\left|
				f(X)
				-
				f(\widehat{X})
			\right|^{q}
			\1_{\Omega(f,\lambda)^{{\rm c}}}
		\right] \notag
		\\&\leq
		(2\lambda)^{q-1}
		\int_{-\lambda}^{\lambda}
			\left(
				(2K_{0})^{p}
				+
				A_{1}
				\{
					\|p_{X}\|_{\infty}
					+
					\|p_{\widehat{X}}\|_{\infty}
				\}
				\int_{\real^{d}}
					|D\1_{E_{t}}|
			\right)
		\rd t
		\e\left[
			\left|
				X
				-
				\widehat{X}
			\right|^{p}
		\right]^{\frac{1}{p+1}} \notag\\
		&\leq
		(2\lambda)^{q-1}
		\left(
			2^{p+1}K_{0}^{p}
			\lambda
			+
			A_{1}
			\{
				\|p_{X}\|_{\infty}
				+
				\|p_{\widehat{X}}\|_{\infty}
			\}
			\int_{\real^{d}}
				|Df|
		\right)
		\e\left[
			\left|
				X
				-
				\widehat{X}
			\right|^{p}
		\right]^{\frac{1}{p+1}}.
	\end{align}
	If $r=\infty$ (i.e., $f$ is
	essentially bounded with respect to Lebesgue measure), then since $X$ and $\widehat{X}$ have density functions, by choosing $\lambda:=\|f\|_{\infty}$, it holds that $\p(\Omega(f,\lambda)^{{\rm c}})=1$.
	Thus the estimate \eqref{eq_0_2} implies that the estimate \eqref{main_0_1} in the case of $r=\infty$ holds.
	
	We next show the estimate \eqref{main_0_1} in the case of $r \in (1,\infty)$ and ${\mathbb E}[|X-\widehat{X}|^{p}]<1$.
	On the event $\Omega(f,\lambda)$, by using H\"older's inequality with $\frac{1}{r/q}+\frac{1}{r/(r-q)}=1$, we obtain
	\begin{align}\label{eq_0_3}
		\e\left[
			\left|
				f(X)
				-
				f(\widehat{X})
			\right|^{q}
			\1_{\Omega(f,\lambda)}
		\right]
		&\leq
		\left(\|f\|_{L^{r}({\mathbb R}^{d},p_{X})}+\|f\|_{L^{r}({\mathbb R}^{d},p_{\widehat{X}})}\right)^{q}
		\p(\Omega(f,\lambda))^{1-\frac{q}{r}} \notag\\
		&\leq
		\left(\|f\|_{L^{r}({\mathbb R}^{d},p_{X})}+\|f\|_{L^{r}({\mathbb R}^{d},p_{\widehat{X}})}\right)^{r}
		\lambda^{-(r-q)}.
	\end{align}
	We choose $\lambda:=(\e[|X-\widehat{X}|^{p}]^{\frac{1}{p+1}})^{-1/r}>1$.
	Then by \eqref{eq_0_2} and \eqref{eq_0_3}, we have
	\begin{align*}
		{\mathbb E}
		\left[
			\left|
				f(X)
				-
				f(\widehat{X})
			\right|^{q}
		\right]
		&\leq 
			2^{q-1}\lambda^{q}
			\left(
				2^{p+1}K_{0}^{p}
				+
				A_{1}
				\{
					\|p_{X}\|_{\infty}
					+
					\|p_{\widehat{X}}\|_{\infty}
				\}
				\int_{{\mathbb R}^{d}}|Df|
			\right)
			{\mathbb E}
			\left[
				\left|
					X
					-
					\widehat{X}
				\right|^{p}
			\right]^{\frac{1}{p+1}}\\
		&\hspace{0.35cm}
		+
		\left(
			\|f\|_{L^{r}({\mathbb R}^{d},p_{X})}
			+
			\|f\|_{L^{r}({\mathbb R}^{d},p_{\widehat{X}})}
		\right)^{r}
		\lambda^{-(r-q)} \\
		&=C_{BV}(p,q,r)
		\e\left[
			\left|
				X
				-
				\widehat{X}
			\right|^{p}
		\right]^{\frac{1-q/r}{p+1}},
	\end{align*}
	which concludes the estimate \eqref{main_0_1} in the case of $r \in (1,\infty)$ and ${\mathbb E}[|X-\widehat{X}|^{p}]<1$.
	In the case of $r \in (1,\infty)$ and ${\mathbb E}[|X-\widehat{X}|^{p}] \geq 1$, it is already shown in Remark \ref{rem:2.12} (iv).
\end{proof}

\subsubsection*{Orlicz--Sobolev spaces and Sobolev spaces with variable exponents}
For a function $f$ in $BV_{\mathrm{loc}}(\real^{d})$ or $W^{1,1}_{\mathrm{loc}}(\real^{d})$, $\int_{\real^{d}}|Df|$ might not be finite, and thus it is difficult to estimate the probability $\p(M(Df)(X)>\lambda)$.
Therefore, we now consider Avikainen's estimates for several subspaces of $W^{1,1}_{\mathrm{loc}}(\real^{d})$.

We first consider the case of the Orlicz--Sobolev space.

\begin{Thm}\label{main_1}
	Let $\Phi$ be an $N$-function and $\Psi$ be its complementary function.
	Suppose that $\Psi$ satisfies the $\Delta_{2}$-condition.
	Let $X, \widehat{X}:\Omega \to \real^{d}$ be random variables which admit density functions $p_{X}$ and $p_{\widehat{X}}$ with respect to Lebesgue measure, respectively, and let $r \in (1,\infty]$.
	Suppose that $p_{X}, p_{\widehat{X}} \in  L^{\infty}(\real^{d}) $ or  $p_{X}, p_{\widehat{X}} \in  L^{\Psi}(\real^{d})$.
	Then for any $f \in W^{1,\Phi}(\real^{d}) \cap L^{r}(\real^{d},p_{X}) \cap L^{r}(\real^{d},p_{\widehat{X}})$ and $q \in (0,r)$, it holds that
\begin{align}\label{main_1_1}
	&\e
	\left[
		\left|
			f(X)
			-
			f(\widehat{X})
		\right|^{q}
	\right] \notag \\
	&\leq
	\left\{
	\begin{array}{ll}
	\displaystyle
	C_{W^{1,\Phi}}(q,r,\infty)
	\inf_{\lambda>0}
	\left\{
		\lambda^{-(1-\frac{q}{r})}
		+
		(\Phi^{-1}(\lambda))^{q}
		{\mathbb E}
		\left[
			\left|
				X-\widehat{X}
			\right|^{q}
		\right]
	\right\},
	&\text{if} \quad p_{X}, p_{\widehat{X}} \in L^{\infty}(\real^{d}), \\
	\displaystyle
	C_{W^{1,\Phi}}(q,r,\Psi)
	\e\left[
		\left|
			X
			-
			\widehat{X}
		\right|^{q}
	\right]^{\frac{1-q/r}{q+1-q/r}},
	&\text{if} \quad p_{X}, p_{\widehat{X}} \in L^{\Psi}(\real^{d}),
	\end{array}\right.
\end{align}
where the constants $C_{W^{1,\Phi}}(q,r,\infty)$ and $C_{W^{1,\Phi}}(q,r,\Psi)$ are defined by
\begin{align*}
	&C_{W^{1,\Phi}}(q,r,\infty)
	\\&:=
	\max\left\{
		\left(
			\frac{2K_{0}}{\alpha}
		\right)^{q},
		\left(
			\|f\|_{L^{r}(\real^{d},p_{X})}
			+
			\|f\|_{L^{r}(\real^{d},p_{\widehat{X}})}
		\right)^{q}
		\left(
			A_{1}
			\{
				\|p_{X}\|_{\infty}
				+
				\|p_{\widehat{X}}\|_{\infty}
			\}
			\|\Phi(\alpha|Df|)\|_{L^{1}({\mathbb R}^{d})}
		\right)^{1-\frac{q}{r}}
	\right\}, \\
	&C_{W^{1,\Phi}}(q,r,\Psi)
	\\&:=
	(2K_{0})^{q}
	+
	\left(
		\|f\|_{L^{r}(\real^{d},p_{X})}
		+
		\|f\|_{L^{r}(\real^{d},p_{\widehat{X}})}
	\right)^{q}
	\left(
		2A_{\Phi}
		\{
			\|p_{X}\|_{L^{\Psi}({\mathbb R}^{d})}
			+
			\|p_{\widehat{X}}\|_{L^{\Psi}({\mathbb R}^{d})}
		\}
		\| |Df| \|_{L^{\Phi}({\mathbb R}^{d})}
	\right)^{1-\frac{q}{r}}.
\end{align*}
Here, $K_{0}$, $A_{1}$ and $A_{\Phi}$ are the constants of the pointwise estimate \eqref{Lem_key_0_1} in Lemma \ref{Lem_key_0}, of the Hardy--Littlewood maximal weak and strong type estimates in Lemma \ref{Lem_key_2} (i) and Lemma \ref{Lem_key_3}, respectively, and $\alpha$ is any positive constant such that $\|\Phi(\alpha|Df|)\|_{L^{1}({\mathbb R}^{d})}<\infty$.
\end{Thm}

\begin{Rem}\label{rem:2.13}
	\begin{itemize}
		\item[(i)]
		In the right hand side of the estimates \eqref{main_1_1}, the power inside of the expectation is $q$ not $p \in (0,\infty)$ unlike the case of $BV(\real^{d})$ (see, Theorem \ref{main_0}).
		The reason is that we do not know the indicator function $\1_{E_{t}}$, $E_{t}:=\{x \in \real^{d}~;~f(x) > t\}$ belongs to $BV(\real^{d})$ or $W^{1,\Phi}(\real^{d})$, and thus we cannot apply the trick \eqref{Lem_key_1_1} for replacing the power $q$ by $p$.
		
		\item[(ii)]
		Let $\Phi$ be a Young function.
		Since $W^{1,\Phi}(\real^{d}) \subset BV_{\mathrm{loc}}(\real^{d})$ (see, Remark \ref{Rem_Orlicz_1} (iv)), the pointwise estimates in Lemma \ref{Lem_key_0} and Remark \ref{Rem_0} hold for $f \in W^{1,\Phi}({\mathbb R}^{d})$.
		Moreover, by using Jensen's inequality for the convex function $\Phi$, for $\mathrm{Leb}$-a.e. $x,y \in {\mathbb R}^{d}$,
		\begin{align}\label{eq:21}
			|f(x)-f(y)|
			\leq
			K_{1}|x-y|
			\left\{
				\Phi^{-1}(M_{2|x-y|}(\Phi(|Df|))(x))
				+
				\Phi^{-1}(M_{2|x-y|}(\Phi(|Df|))(y))
		\right\}.
		\end{align}
		%and
		%\begin{align}\label{eq:22}
		%	|f(x)-f(y)|
		%	\leq
		%	K_{1}|x-y|
		%	\left\{
		%		\Phi^{-1}(M_{2|x-y|}(\Phi(|Df|))(x))
		%		+
		%		\Phi^{-1}(M_{2|x-y|}(\Phi(|Df|))(y))
		%	\right\}.
		%\end{align}
		Theorem 1.2 in \cite{Tu07} shows that functions $f \in W^{1,\Phi}(\real^{d})$ can be characterized by the estimate \eqref{eq:21}.
		
		\item[(iii)]
		If $r \in (1,\infty)$, then the constants $C_{W^{1,\Phi}}(q,r,\infty)$ and $C_{W^{1,\Phi}}(q,r,\Psi)$ depend on both densities $p_{X}$ and $p_{\widehat{X}}$ or expectations $\e[|f(X)|^{r}]$ and $\e[|f(\widehat{X})|^{r}]$ (see, Remark \ref{rem:2.12} (v)) .
\end{itemize}
\end{Rem}

As a conclusion of Theorem \ref{main_0} and Theorem \ref{main_1} noting Example \ref{Ex_Orlicz_0} (i), we obtain the following estimates for the Sobolev space $W^{1,p}(\real^{d})$ for $p \in [1,\infty)$.

\begin{Cor}\label{Cor_1}
	Let $X, \widehat{X}:\Omega \to \real^{d}$ be random variables which admit density functions $p_{X}$ and $p_{\widehat{X}}$ with respect to Lebesgue measure, respectively, and let $r \in (1,\infty]$, $p \in [1,\infty)$ and $p^{*}:=p/(p-1)$.
	Suppose that $p_{X}, p_{\widehat{X}} \in  L^{\infty}(\real^{d})$ or $p_{X},p_{\widehat{X}} \in  L^{p^{*}}(\real^{d})$.
	Then for any $f \in W^{1,p}(\real^{d}) \cap L^{r}(\real^{d},p_{X}) \cap L^{r}(\real^{d},p_{\widehat{X}})$ and $q \in (0,r)$, there exist $C_{W^{1,p}}(q,r,\infty)>0$ and $C_{W^{1,p}}(q,r,p^{*})>0$ such that
	\begin{align*}
		\e\left[
			\left|
				f(X)
				-
				f(\widehat{X})
			\right|^{q}
		\right]
		&\leq
		\left\{ \begin{array}{ll}
		\displaystyle
			C_{W^{1,p}}(q,r,\infty)
			\e\left[
				\left|
					X
					-
					\widehat{X}
				\right|^{q}
			\right]^{\frac{p(1-q/r)}{q+p(1-q/r)}},
			&\text{if} \quad p_{X}, p_{\widehat{X}} \in L^{\infty}(\real^{d}),\\
		\displaystyle
			C_{W^{1,p}}(q,r,p^{*})
			\e\left[
				\left|
					X
					-
					\widehat{X}
				\right|^{q}
			\right]^{\frac{1-q/r}{q+1-q/r}},
		&\text{if} \quad p_{X}, p_{\widehat{X}} \in L^{p^{*}}(\real^{d}).
		\end{array}\right.
	\end{align*}
\end{Cor}
\begin{proof}
		It is sufficient to consider $p_{X}, p_{\widehat{X}} \in L^{\infty}(\real^{d})$.
		Let $\Phi(x):=x^{p}/p$, $x \in [0,\infty)$.
		Then the Orlicz--Sobolev space $W^{1,\Phi}(\real^{d})$ coincides with the classical Sobolev space $W^{1,p}(\real^{d})$ (see, Example \ref{Ex_Orlicz_0} (i)).
		Since $\Phi^{-1}(x)=p^{1/p} x^{1/p}$, the infimum of the right hand side of \eqref{main_1_1} is bounded from above by
		\begin{align*}
			\lambda^{-(1-\frac{q}{r})}
			+
			(p \lambda)^{q/p}
			{\mathbb E}
			\left[
				\left|
					X-\widehat{X}
				\right|^{q}
			\right]
		\end{align*}
		for any $\lambda >0$.
		By choosing $\lambda:=\e[|X-\widehat{X}|^{q}]^{-\frac{1}{1+q/p-q/r}}$, we conclude the statement.
\end{proof}

\begin{Rem}
	Let $f \in W^{1,p}(\real^{d})$ with $d<p<\infty$.
	Then by using Morrey's inequality (see, Theorem 4.10 in \cite{EvGa92}) and Lebesgue differentiation theorem (see, Theorem 1.32 in \cite{EvGa92}) there exists a $(1-d/p)$-H\"older continuous function $f^{*}$ such that $f=f^{*}$,  $\mathrm{Leb}$-a.e.
	Hence from Jensen's inequality, we have
	\begin{align*}
		\e\left[
			\left|
				f(X)
				-
				f(\widehat{X})
			\right|^{q}
		\right]
		\leq
		\|f^{*}\|_{1-d/p}^{q}
		\e\left[
			\left|
				X
				-
				\widehat{X}
			\right|^{q}
		\right]^{1-\frac{d}{p}},
	\end{align*}
	where $\|f^{*}\|_{\alpha}:=\sup_{x \neq y}
	\frac{|f^{*}(x)-f^{*}(y)|}{|x-y|^{\alpha}}$ for $\alpha \in (0,1]$.
	On the other hand, if $\e[|X-\widehat{X}|^{q}]<1$ and $q \in (0,d/(p-d))$ i.e., $1-d/p<1/(q+1)$, then Corollary 2.19 with $p_{X} \in L^{p^{*}}(\real^{d})$ and $r= \infty$ is sharper than the above.
\end{Rem}

\begin{proof}[Proof of Theorem \ref{main_1}]
We first assume that $p_{X}, p_{\widehat{X}} \in L^{\infty}(\real^{d})$.
Since $|Df| \in L^{\Phi}({\mathbb R}^{d})$, there exists $\alpha>0$ such that $\|\Phi(\alpha|Df|)\|_{L^{1}({\mathbb R}^{d})}<\infty$.
Then for $\lambda>0$, we define the event $\Omega(\Phi(\alpha|Df|),\lambda) \in {\mathscr F}$ by
\begin{align*}
	\Omega(\Phi(\alpha|Df|),\lambda)
	:=
	\left\{
		M(\Phi(\alpha|Df|))(X)>\lambda
	\right\}
	\cup
	\left\{
		M(\Phi(\alpha|Df|))(\widehat{X})>\lambda
	\right\}.
\end{align*}
Since $X$ and $\widehat{X}$ have bounded density functions, by using Lemma \ref{Lem_key_2} (i), we obtain
\begin{align*}
	\p(\Omega(\Phi(\alpha|Df|),\lambda))
	&\leq
	A_{1}
	\{
		\|p_{X}\|_{\infty}
		+
		\|p_{\widehat{X}}\|_{\infty}
	\}
	\|\Phi(\alpha|Df|)\|_{L^{1}({\mathbb R}^{d})}
	\lambda^{-1}.
\end{align*}
Hence by using H\"older's inequality with $\frac{1}{r/q}+\frac{1}{r/(r-q)}=1$ in the case of $r \in (1,\infty)$ and by using the boundedness of $f$ in the case of $r=\infty$, we have
\begin{align}\label{eq:14}
	&{\mathbb E}
	\left[
		\left|
		f(X)
		-
		f(\widehat{X})
	\right|^{q}
	{\bf 1}_{\Omega(\Phi(\alpha|Df|),\lambda)}
	\right] \notag \\ \notag
	&\leq
	\left(
		\|f\|_{L^{r}({\mathbb R}^{d},p_{X})}
		+
		\|f\|_{L^{r}({\mathbb R}^{d},p_{\widehat{X}})}
	\right)^{q}
	{\mathbb P}(\Omega(\Phi(\alpha|Df|),\lambda))^{1-\frac{q}{r}}\\
	&\leq
	\left(
		\|f\|_{L^{r}({\mathbb R}^{d},p_{X})}
		+
		\|f\|_{L^{r}({\mathbb R}^{d},p_{\widehat{X}})}
	\right)^{q}
	\left(
		A_{1}
		\{
			\|p_{X}\|_{\infty}
			+
			\|p_{\widehat{X}}\|_{\infty}
		\}
		\|\Phi(\alpha|Df|)\|_{L^{1}({\mathbb R}^{d})}
	\right)^{1-\frac{q}{r}}
	\lambda^{-(1-\frac{q}{r})}.
\end{align}
Let $N \in {\mathscr B}({\mathbb R}^{d})$ be the Lebesgue null set defined on Lemma \ref{Lem_key_0}.
On the event $\Omega(\Phi(\alpha|Df|),\lambda)^{{\rm c}}$, since $X, \widehat{X}$ have density functions and $\Phi^{-1}$ is non-decreasing, by similar way as   \eqref{eq:21} in Remark \ref{rem:2.13} (ii), we obtain
\begin{align}\label{eq:15}
	&{\mathbb E}
	\left[
		\left|
			f(X)
			-
			f(\widehat{X})
		\right|^{q}
		{\bf 1}_{\Omega(\Phi(\alpha|Df|),\lambda)^{{\rm c}}}
	\right] \notag \\
	&=
	{\mathbb E}
	\left[
		\left|
			f(X)
			-
			f(\widehat{X})
		\right|^{q}
		{\bf 1}_{\Omega(\Phi(\alpha|Df|),\lambda)^{{\rm c}}}
		\1_{\real^{d} \setminus N}(X)
		\1_{\real^{d} \setminus N}(\widehat{X})
	\right]\notag\\
	&\leq
	\left(\frac{K_{0}}{\alpha}\right)^{q}
	{\mathbb E}
	\left[
		\left|
			X
			-
			\widehat{X}
		\right|^{q}
		\left\{
			\Phi^{-1}(M(\Phi(\alpha|Df|))(X))
			+
			\Phi^{-1}(M(\Phi(\alpha|Df|))(\widehat{X}))
		\right\}^{q}
		{\bf 1}_{\Omega(\Phi(\alpha|Df|),\lambda)^{{\rm c}}}
	\right]\notag\\
	&\leq
	\left(\frac{2K_{0}}{\alpha}\right)^{q}
	(\Phi^{-1}(\lambda))^{q}
	{\mathbb E}
	\left[
	\left|
	X
	-
	\widehat{X}
	\right|^{q}
	\right].
\end{align}
Hence, by \eqref{eq:14} and \eqref{eq:15}, we have
\begin{align*}
	{\mathbb E}
	\left[
		\left|
			f(X)
			-
			f(\widehat{X})
		\right|^{q}
	\right]
	\leq
	C_{W^{1,\Phi}}(q,r,\infty)
	\left(
		\lambda^{-1+\frac{q}{r}}
		+
		(\Phi^{-1}(\lambda))^{q}
		{\mathbb E}
		\left[
			\left|
				X
				-
				\widehat{X}
			\right|^{q}
		\right]
	\right),
\end{align*}
which concludes the statement for $p_{X} \in L^{\infty}(\real^{d})$.

Now we suppose $p_{X}, p_{\widehat{X}} \in L^{\Psi}(\real^{d})$.
For $\lambda>0$, we define the event $\Omega(Df,\lambda) \in \mathscr{F}$ by
\begin{align*}
	\Omega(Df,\lambda)
	:=
	\left\{
		M(Df)(X)>\lambda
	\right\}
	\cup
	\left\{
		M(Df)(\widehat{X})>\lambda
	\right\}.
\end{align*}
Since $\Psi$ satisfies the $\Delta_{2}$-condition, it follows from Lemma \ref{Lem_key_3} that
\begin{align*}
	\|M(Df)\|_{L^{\Phi}({\mathbb R}^{d})}
	\leq
	A_{\Phi}
	\||Df|\|_{L^{\Phi}({\mathbb R}^{d})}.
\end{align*}
Hence by using the Markov inequality and the generalized H\"older's inequality \eqref{eq_GHolder}, we obtain
\begin{align*}
	{\mathbb P}(\Omega(Df,\lambda))
	&\leq
	\int_{{\mathbb R}^{d}}
		M(Df)(x)
		\{
			p_{X}(x)
			+
			p_{\widehat{X}}(x)
		\}
	\rd x
	\lambda^{-1}\\
	&\leq
	2
	\|M(Df)\|_{L^{\Phi}({\mathbb R}^{d})}
	\{
		\|p_{X}\|_{L^{\Psi}({\mathbb R}^{d})}
		+
		\|p_{\widehat{X}}\|_{L^{\Psi}({\mathbb R}^{d})}
	\}
	\lambda^{-1}\\
	&\leq
	2A_{\Phi}
	\||Df|\|_{L^{\Phi}({\mathbb R}^{d})}
	\{
		\|p_{X}\|_{L^{\Psi}({\mathbb R}^{d})}
		+
		\|p_{\widehat{X}}\|_{L^{\Psi}({\mathbb R}^{d})}
	\}
	\lambda^{-1}.
\end{align*}
Hence by using H\"older's inequality with $\frac{1}{r/q}+\frac{1}{r/(r-q)}=1$ in the case of $r \in (1,\infty)$ and by using the boundedness of $f$ in the case of $r=\infty$, we have
\begin{align}\label{eq:25}
	&{\mathbb E}
	\left[
		\left|
			f(X)
			-
			f(\widehat{X})
		\right|^{q}
		{\bf 1}_{\Omega(Df,\lambda)}
	\right] \notag \\
	&\leq
	\left(
		\|f\|_{L^{r}({\mathbb R}^{d},p_{X})}
		+
		\|f\|_{L^{r}({\mathbb R}^{d},p_{\widehat{X}})}
	\right)^{q}
	{\mathbb P}(\Omega(Df,\lambda))^{1-\frac{q}{r}}, \notag \\
	&\leq 
	\left(\|f\|_{L^{r}({\mathbb R}^{d},p_{X})}+\|f\|_{L^{r}({\mathbb R}^{d},p_{\widehat{X}})}\right)^{q}
	\left(
		2
		A_{\Phi}
		\{
			\|p_{X}\|_{L^{\Psi}({\mathbb R}^{d})}
			+
			\|p_{\widehat{X}}\|_{L^{\Psi}({\mathbb R}^{d})}
		\}
		\||Df|\|_{L^{\Phi}({\mathbb R}^{d})}
	\right)^{1-\frac{q}{r}}
	\lambda^{-(1-\frac{q}{r})}.
\end{align}
On the event $\Omega(Df,\lambda)^{{\rm c}}$, since $X$ and $\widehat{X}$ have density functions, by Lemma \ref{Lem_key_0} and Remark \ref{Rem_Orlicz_1} (iv), we obtain
\begin{align}\label{eq:23}
	{\mathbb E}
	\left[
		\left|
			f(X)
			-
			f(\widehat{X})
		\right|^{q}
		{\bf 1}_{\Omega(Df,\lambda)^{{\rm c}}}
	\right]
	&=
	{\mathbb E}
	\left[
		\left|
			f(X)
			-
			f(\widehat{X})
		\right|^{q}
		{\bf 1}_{\Omega(Df,\lambda)^{{\rm c}}}
		\1_{\real^{d} \setminus N}(X)
		\1_{\real^{d} \setminus N}(\widehat{X})
	\right]
	\notag\\
	&\leq
	K_{0}^{q}
	{\mathbb E}
	\left[
		\left|
			X
			-
			\widehat{X}
		\right|^{q}
		\{
			M(Df)(X)
			+
			M(Df)(\widehat{X})
		\}^{q}
		{\bf 1}_{\Omega(Df,\lambda)^{{\rm c}}}
	\right]\notag\\
	&\leq
	(2K_{0})^{q}
	\lambda^{q}
	{\mathbb E}
	\left[
		\left|
			X
			-
			\widehat{X}
		\right|^{q}
	\right].
\end{align}
We choose $\lambda:=\e[|X-\widehat{X}|^{q}]^{-\frac{1}{q+1}}$ in the case of $r=\infty$ and $\lambda:=\e[|X-\widehat{X}|^{q}]^{-\frac{1}{q+1-q/r}}$ in the case of $r \in (1,\infty)$, and then we conclude the statement from \eqref{eq:25} and \eqref{eq:23}.
\end{proof}

We next consider the case of the Sobolev space with a variable exponent.

\begin{Thm}\label{main_2}
	Let $p \in {\mathcal P}^{\log}({\mathbb R}^{d})$ with $1<p^{-}$ and $p^{*}(\cdot):=p(\cdot)/(p(\cdot)-1)$.
	Let $X, \widehat{X}:\Omega \to \real^{d}$ be random variables which admit density functions $p_{X}$ and $p_{\widehat{X}}$ with respect to Lebesgue measure, respectively, and let $r \in (1,\infty]$.
	Suppose that $p_{X}, p_{\widehat{X}} \in L^{p^{*}(\cdot)}({\mathbb R}^{d})$.
	Then for any $f \in W^{1,p(\cdot)}(\real^{d}) \cap L^{r}(\real^{d},p_{X}) \cap L^{r}(\real^{d},p_{\widehat{X}})$ and $q \in (0,r)$, it holds that
	\begin{align*}
		{\mathbb E}
		\left[
			\left|
				f(X)
				-
				f(\widehat{X})
			\right|^{q}
		\right]
		&\leq
		C_{W^{1,p(\cdot)}}(q,r,p^{*}(\cdot))
		{\mathbb E}
		\left[
			\left|
				X
				-
				\widehat{X}
			\right|^{q}
		\right]^{\frac{1-q/r}{q+1-q/r}},
	\end{align*}
	where the constant $C_{W^{1,p(\cdot)}}(q,r,p^{*}(\cdot))$ is defined by
\begin{align*}	
	&C_{W^{1,p(\cdot)}}(q,r,p^{*}(\cdot))
	\\&:=
	(2K_{0})^{q}
	+
	\left(
		\|f\|_{L^{r}(\real^{d},p_{X})}
		+
		\|f\|_{L^{r}(\real^{d},p_{\widehat{X}})}
	\right)^{q}
	\left(
		2A_{\Phi}
		\{
			\|p_{X}\|_{L^{p^{*}(\cdot)}({\mathbb R}^{d})}
			+
			\|p_{\widehat{X}}\|_{L^{p^{*}(\cdot)}({\mathbb R}^{d})}
		\}
		\| |Df| \|_{L^{p(\cdot)}({\mathbb R}^{d})}
	\right)^{1-\frac{q}{r}}.
\end{align*}
Here, $K_{0}$ and $A_{p(\cdot)}$ are the constants of the pointwise estimate \eqref{Lem_key_0_1} in Lemma \ref{Lem_key_0} and of the Hardy--Littlewood maximal strong estimate in Lemma \ref{lem:0.3}, respectively.
\end{Thm}
\begin{proof}
We can use the Hardy--Littlewood maximal strong type estimate in Lemma \ref{lem:0.3} since $p \in \mathcal{P}^{\log}(\real^{d})$.
Moreover, it holds that $W^{1,p(\cdot)}({\mathbb R}^{d}) \subset BV_{{\rm loc}}({\mathbb R}^{d})$ (see, Remark \ref{Rem_Orlicz_1} (iv) and Remark \ref{Rem_Sob_expo} (iv)).
Therefore, by the same way as the proof of Theorem \ref{main_1}, we can prove the statement by using the generalized H\"older's inequality \eqref{eq_GHolder_1} for $M(Df) \in L^{p(\cdot)}({\mathbb R}^{d})$ and $p_{X}, p_{\widehat{X}} \in L^{p^{*}(\cdot)}({\mathbb R}^{d})$, and thus it will be omitted.
\end{proof}

\begin{Rem}
For the Sobolev space with a variable exponent $p$, it is difficult to obtain Avikainen's estimate in the case of $p_{X}, p_{\widehat{X}} \in L^{\infty}({\mathbb R}^{d})$.
The reason is that since the variable exponent $p$ is not constant, we cannot use Jensen's inequality in the same way as the estimate \eqref{eq:15}.
\end{Rem}

\subsubsection*{Fractional Sobolev spaces}

We finally consider Avikainen's estimates for fractional Sobolev spaces.

\begin{Thm}\label{main_3}
	Let $s \in (0,1)$, $p \in [1,\infty)$ and $p^{*}:=p/(p-1)$.
	Let $X, \widehat{X}:\Omega \to \real^{d}$ be random variables which admit the density functions $p_{X}$ and $p_{\widehat{X}}$ with respect to Lebesgue measure, respectively, and let $r \in (1,\infty]$.
	Suppose that $p_{X}, p_{\widehat{X}} \in  L^{\infty}(\real^{d})$ or $p_{X}, p_{\widehat{X}} \in L^{p^{*}}({\mathbb R}^{d})$.
	Then for any $f \in W^{s,p}(\real^{d}) \cap L^{r}(\real^{d},p_{X}) \cap L^{r}(\real^{d},p_{\widehat{X}})$ and $q \in (0,r)$, it holds that
	\begin{align*}
		{\mathbb E}
		\left[
			\left|
				f(X)
				-
				f(\widehat{X})
			\right|^{q}
		\right]
		\leq
		\left\{ \begin{array}{ll}
		\displaystyle
			C_{W^{s,p}}(q,r,\infty)
			{\mathbb E}
			\left[
				\left|
					X
					-
					\widehat{X}
				\right|^{qs}
			\right]^{\frac{p(1-q/r)}{q+p(1-q/r)}},
			&\text{if} \quad p_{X}, p_{\widehat{X}} \in L^{\infty}(\real^{d}), \\
		\displaystyle 
			C_{W^{s,p}}(q,r,p^{*})
			{\mathbb E}
			\left[
				\left|
					X
					-
					\widehat{X}
				\right|^{qs}
			\right]^{\frac{1-q/r}{q+1-q/r}},
			&\text{if} \quad p_{X}, p_{\widehat{X}} \in L^{p^{*}}(\real^{d}),
		\end{array}\right.
	\end{align*}
	where
	\begin{align*}
		&C_{W^{s,p}}(q,r,\infty)
		\\&:=
		(2K_{0}(s,p))^{q}
		+
		\left(
			\|f\|_{L^{r}(\real^{d},p_{X})}
			+
			\|f\|_{L^{r}(\real^{d},p_{\widehat{X}})}
		\right)^{q}
		\left(
			A_{1}
			\{
				\|p_{X}\|_{\infty}
				+
				\|p_{\widehat{X}}\|_{\infty}
			\}
			\||G_{s,p}f|^{p}\|_{L^{1}({\mathbb R}^{d})}
		\right)^{1-\frac{q}{r}}, \\
		&C_{W^{s,p}}(q,r,p^{*})
		\\&:=
		(2K_{0}(s,p))^{q}
		+
		\left(
			\|f\|_{L^{r}(\real^{d},p_{X})}
			+
			\|f\|_{L^{r}(\real^{d},p_{\widehat{X}})}
		\right)^{q}
		\left(
			A_{p}
			\{
				\|p_{X}\|_{L^{p^{*}}(\real^{d})}
				+
				\|p_{\widehat{X}}\|_{L^{p^{*}}(\real^{d})}
			\}
			\|G_{s,p}f\|_{L^{p}({\mathbb R}^{d})}
		\right)^{1-\frac{q}{r}}.
	\end{align*}
Here, $K_{0}(s,p)$, $A_{1}$ and $A_{p}$ are the constants of pontwise estimate \eqref{eq:28} in Lemma \ref{lem:0.5} and of the Hardy--Littlewood maximal weak and strong type estimates in Lemma \ref{Lem_key_2} (i) and (ii), respectively.
\end{Thm}

Before proving Theorem \ref{main_3}, we give a pointwise estimate for functions in $W^{s,p}(\real^{d})$, which plays a crucial role in our argument.

\begin{Lem}\label{lem:0.5}
	Let $s \in (0,1)$, $p \in [1,\infty)$ and $f \in W^{s,p}({\mathbb R}^{d})$.
	Then there exist a constant $K_{0}(s,p)>0$ and a Lebesgue null set $N \in {\mathscr B}({\mathbb R}^{d})$ such that for all $x,y \in {\mathbb R}^{d} \setminus N$,
	\begin{align}
	\label{eq:28}
		\left|
			f(x)
			-
			f(y)
		\right|
		\leq
		K_{0}(s,p)
		|x-y|^{s}
		\left\{
			M_{2|x-y|}(G_{s,p}f)(x)
			+
			M_{2|x-y|}(G_{s,p}f)(y)
		\right\}.
	\end{align}
\end{Lem}

\begin{Rem}
	%As a consequence of Lemma \ref{lem:0.5}, the following pointwise estimate holds:
	%\begin{align}\label{pw_esti_2}
	%	|f(x)-f(y)|
	%	&\leq
	%	K_{0}(s,p)
	%	|x-y|^{s}
	%	\left\{
	%		M_{2|x-y|}(G_{s,p}f)(x)
	%		+
	%		M_{2|x-y|}(G_{s,p}f)(y)
	%	\right\},
	%\text{$\mathrm{Leb}$-a.e. }x,y \in \real^{d}
	%\end{align}
	%for some $K_{1}(s,p)>0$.
	Note that Yang \cite{Ya03} introduced Haj\l{}asz--Sobolev space $W^{s,p}(X)$ on a metric measure space $X$ of homogeneous type by using the pointwise estimate similar to \eqref{eq:28} (see, Definition 1.4 in \cite{Ya03}).
\end{Rem}

\begin{proof}[Proof of Lemma \ref{lem:0.5}]
	The proof is similar to Lemma \ref{Lem_key_0}.
	By using Jensen's inequality, for any $x \in \real^{d}$ and $r>0$,
	\begin{align}\label{eq:0.1}
		\dashint_{B(x;r)}
			\left|
				f(z)
				-
				(f)_{x,r}
			\right|
		{\rm d}z
		&\leq
		\dashint_{B(x;r)}
			\left(
				\dashint_{B(x;r)}
					\left|
						f(z)
						-
						f(y)
					\right|^{p}
				{\rm d}y
			\right)^{1/p}
		{\rm d}z \notag\\
		&\leq
		(2r)^{(d+sp)/p}
		\dashint_{B(x;r)}
			\left(
				\dashint_{B(x;r)}
					\frac
					{|f(z)-f(y)|^{p}}
					{|z-y|^{d+sp}}
				{\rm d}y
			\right)^{1/p}
		{\rm d}z \notag\\
		&\leq
		C_{0}(s,p)
		r^{s}
		\dashint_{B(x;r)}
			G_{s,p}f(z)
		\rd z \notag\\
		&\leq
		C_{0}(s,p)
		r^{s}
		M_{r}(G_{s,p}f)(x),
	\end{align}
	where $C_{0}(s,p):=2^{(d+sp)/p}(\frac{\Gamma(d/2+1)}{\pi^{d/2}})^{1/p}$.
        Let $N \in {\mathscr B}({\mathbb R}^{d})$ be the Lebesgue null set defined on \eqref{Lem_key_0_3}.
	Then, by the same way as the proof of Lemma \ref{Lem_key_0}, for fixed $x, y \in \real^{d} \setminus N$ and $r_{i}:=2^{-i}|x-y|$, for $i \in \n \cup \{0\}$, we obtain
	\begin{align}\label{eq:0.2}
		\left|
			f(x)
			-
			(f)_{x,r_{0}}
		\right|
		&\leq
		2^{d}
		\sum_{i=0}^{\infty}
		\dashint_{B(x;r_{i})}
			\left|
				f(z)
				-
				(f)_{x,r_{i}}
			\right|
		{\rm d}z \notag\\
		&\leq 
		2^{d}
		C_{0}(s,p)
		M_{|x-y|}(G_{s,p}f)(x)
		\sum_{i=0}^{\infty}
		r_{i}^{s}
		\notag\\
		&=
		\frac{2^{s+d}}{2^{s}-1}
		C_{0}(s,p)
		M_{|x-y|}(G_{s,p}f)(x)
		|x-y|^{s}.
	\end{align}
	By the same way,
	%since $B(y;r_{0}) \subset B(x;2r_{0})$,
	we have
	\begin{align}\label{eq:0.22}
		\left|
			f(y)
			-
			(f)_{y,r_{0}}
		\right|
		&\leq
		\frac{2^{s+d}}{2^{s}-1}
		C_{0}(s,p)
		M_{|x-y|}(G_{s,p}f)(y)
		|x-y|^{s}.
		%\notag\\
		%&\leq
		%\frac{2^{s+2d}}{2^{s}-1}
		%C_{0}(s,p)
		%M_{2|x-y|}(G_{s,p}f)(x)
		%|x-y|^{s}.
	\end{align}
	On the other hand, by the same way as proof of Lemma \ref{Lem_key_0}, it holds from \eqref{eq:0.1} that
	\begin{align}\label{eq:0.23}
		|(f)_{x,r_{0}}-(f)_{y,r_{0}}|
		&\leq
		2^{d+1}
		\dashint_{B(x;2r_{0})}
		|f(z)-(f)_{x,2r_{0}}|
		\rd z \notag\\
		&\leq
		2^{s+d+1}
		C_{0}(s,p)
		|x-y|^{s}
		M_{2|x-y|}(G_{s,p}f)(x).
	\end{align}
	By combining \eqref{eq:0.2}, \eqref{eq:0.22} and \eqref{eq:0.23}, we conclude the proof.
\end{proof}

\begin{proof}[Proof of Theorem \ref{main_3}]
	We first assume $p_{X},p_{\widehat{X}} \in L^{\infty}(\real^{d})$.
	For $\lambda>0$, we define the event $\Omega(|G_{s,p}f|^{p},\lambda) \in \mathscr{F}$ by
		\begin{align*}
		\Omega(|G_{s,p}f|^{p},\lambda)
		:=
		\left\{
			M(|G_{s,p}f|^{p})(X)
			>
			\lambda
		\right\}
		\cup
		\left\{
			M(|G_{s,p}f|^{p})(\widehat{X})
			>
			\lambda
		\right\}.
	\end{align*}
	Since $|G_{s,p}f|^{p} \in L^{1}({\mathbb R}^{d})$, by using Lemma \ref{Lem_key_2} (i), we obtain
	\begin{align*}
		\p(\Omega(|G_{s,p}f|^{p},\lambda))
		\leq
		A_{1}
		\{
			\|p_{X}\|_{\infty}
			+
			\|p_{\widehat{X}}\|_{\infty}
		\}
		\||G_{s,p}f|^{p}\|_{L^{1}(\real^{d})}
		\lambda^{-1}.
	\end{align*}
	Hence by using H\"older's inequality with $\frac{1}{r/q}+\frac{1}{r/(r-q)}=1$ in the case of $r \in (1,\infty)$ and by using the boundedness of $f$ in the case of $r=\infty$, we have
	\begin{align}\label{eq:18}
		&{\mathbb E}
		\left[
			\left|
				f(X)
				-
				f(\widehat{X})
			\right|^{q}
			{\bf 1}_{\Omega(|G_{s,p}|^{p},\lambda)}
		\right] \notag\\
		&\leq
	\left(
		\|f\|_{L^{r}(\real^{d},p_{X})}
		+
		\|f\|_{L^{r}(\real^{d},p_{\widehat{X}})}
	\right)^{q}
	{\mathbb P}(\Omega(|G_{s,p}f|^{p},\lambda))^{1-\frac{q}{r}} \notag \\
	&\leq
	\left(
		\|f\|_{L^{r}(\real^{d},p_{X})}
		+
		\|f\|_{L^{r}(\real^{d},p_{\widehat{X}})}
	\right)^{q}
	\left(
		A_{1}
		\{
			\|p_{X}\|_{\infty}
			+
			\|p_{\widehat{X}}\|_{\infty}
		\}
		\||G_{s,p}f|^{p}\|_{L^{1}({\mathbb R}^{d})}
	\right)^{1-\frac{q}{r}}
	\lambda^{-(1-\frac{q}{r})}.
\end{align}
Let $N \in {\mathscr B}({\mathbb R}^{d})$ be the Lebesgue null set defined on Lemma \ref{lem:0.5}.
On the event $\Omega(|G_{s,p}f|^{p},\lambda)^{{\rm c}}$, since $X$ and $\widehat{X}$ have density functions, by Lemma \ref{lem:0.5} and using Jensen's inequality, we obtain
\begin{align}
	\label{eq:35}
	&{\mathbb E}
	\left[
		\left|
			f(X)
			-
			f(\widehat{X})
		\right|^{q}
		{\bf 1}_{\Omega(|G_{s,p}f|^{p},\lambda)^{{\rm c}}}
	\right] \notag 
	=
	{\mathbb E}
	\left[
		\left|
			f(X)
			-
			f(\widehat{X})
		\right|^{q}
		{\bf 1}_{\Omega(|G_{s,p}f|^{p},\lambda)^{{\rm c}}}
		{\bf 1}_{\real^{d}\setminus N}(X){\bf 1}_{\real^{d}\setminus N}(\widehat{X})
	\right] \notag \\
	&\leq
	K_{0}(s,p)^{q}
	{\mathbb E}
	\left[
		\left|
			X
			-
			\widehat{X}
		\right|^{qs}
		\{
			M(G_{s,p}f)(X)
			+
			M(G_{s,p}f)(\widehat{X})
		\}^{q}
		{\bf 1}_{\Omega(|G_{s,p}f|^{p},\lambda)^{{\rm c}}}
	\right]
	\notag \\
	&\leq
	K_{0}(s,p)^{q}
	{\mathbb E}
	\left[
		\left|
			X
			-
			\widehat{X}
		\right|^{qs}
		\{
			M(|G_{s,p}f|^{p})(X)^{\frac{1}{p}}
			+
			M(|G_{s,p}f|^{p})(\widehat{X})^{\frac{1}{p}}
		\}^{q}
		{\bf 1}_{\Omega(|G_{s,p}f|^{p},\lambda)^{{\rm c}}}
	\right]
	\notag \\
	&\leq
	(2K_{0}(s,p))^{q}
	\lambda^{\frac{q}{p}}
	{\mathbb E}
	\left[
		\left|
			X
			-
			\widehat{X}
		\right|^{qs}
	\right].
\end{align}
	By choosing $\lambda:={\mathbb E}[|X-\widehat{X}|^{qs}]^{-\frac{1}{q/p+1-q/r}}$, we conclude the statement for $p_{X} \in L^{\infty}(\real^{d})$ from \eqref{eq:18} and \eqref{eq:35}.
	
	Now we suppose $p_{X}, p_{\widehat{X}} \in L^{p^{*}}(\real^{d})$.
	For $\lambda>0$, we define the event $\Omega(G_{s,p}f, \lambda) \in \mathscr{F}$ by
	\begin{align*}
		\Omega(G_{s,p}f, \lambda)
		:=
		\left\{
			M(G_{s,p}f)(X)>\lambda
		\right\}
		\cup
		\left\{
			M(G_{s,p}f)(\widehat{X})>\lambda
		\right\}.
	\end{align*}
	By using the Markov inequality, H\"older's inequality and Lemma \ref{Lem_key_2} (ii), we obtain
	\begin{align*}
		{\mathbb P}(\Omega(G_{s,p}f, \lambda))
		&\leq
		\int_{{\mathbb R}^{d}}
			M(G_{s,p}f)(x)
			\{
				p_{X}(x)
				+
				p_{\widehat{X}}(x)
			\}
		\rd x
		\lambda^{-1}\\
		&\leq
		\{
			\|p_{X}\|_{L^{p^{*}}({\mathbb R}^{d})}
			+
			\|p_{\widehat{X}}\|_{L^{p^{*}}({\mathbb R}^{d})}
		\}
		\|M(G_{s,p}f)\|_{L^{p}({\mathbb R}^{d})}
		\lambda^{-1}\\
		&\leq
		A_{p}
		\{
			\|p_{X}\|_{L^{p^{*}}({\mathbb R}^{d})}
			+
			\|p_{\widehat{X}}\|_{L^{p^{*}}({\mathbb R}^{d})}
		\}
		\|G_{s,p}f\|_{L^{p}({\mathbb R}^{d})}
		\lambda^{-1}.
	\end{align*}
	Hence by using H\"older's inequality with $\frac{1}{r/q}+\frac{1}{r/(r-q)}=1$ in the case of $r \in (1,\infty)$ and by using the boundedness of $f$ in the case of $r=\infty$, we have
	\begin{align}\label{eq:19}
		&{\mathbb E}
		\left[
			\left|
				f(X)
				-
				f(\widehat{X})
			\right|^{q}
			{\bf 1}_{\Omega(G_{s,p}f, \lambda)}
		\right] \notag \\
		&\leq \left(
			\|f\|_{L^{r}(\real^{d},p_{X})}
			+
			\|f\|_{L^{r}(\real^{d},p_{\widehat{X}})}
		\right)^{q}{\mathbb P}(\Omega(G_{s,p}f,\lambda))^{1-\frac{q}{r}} \notag \\
		&\leq 
		\left(
			\|f\|_{L^{r}(\real^{d},p_{X})}
			+
			\|f\|_{L^{r}(\real^{d},p_{\widehat{X}})}
		\right)^{q}
		\left(
			A_{p}
			\{
				\|p_{X}\|_{L^{p^{*}}({\mathbb R}^{d})}
				+
				\|p_{\widehat{X}}\|_{L^{p^{*}}({\mathbb R}^{d})}
			\}
			\|G_{s,p}f\|_{L^{p}({\mathbb R}^{d})}
			\right)^{1-\frac{q}{r}}
		\lambda^{-(1-\frac{q}{r})}.
	\end{align}
	On the event $\Omega(G_{s,p}f, \lambda)^{{\rm c}}$, since $X$ and $\widehat{X}$ have density functions, by Lemma \ref{lem:0.5}, we obtain
	\begin{align}\label{eq:20}
		&{\mathbb E}
		\left[
			\left|
				f(X)
				-
				f(\widehat{X})
			\right|^{q}
			{\bf 1}_{\Omega(G_{s,p}f, \lambda)^{{\rm c}}}
		\right]
		\notag\\&=
		{\mathbb E}
		\left[
			\left|
				f(X)
				-
				f(\widehat{X})
			\right|^{q}
			{\bf 1}_{\Omega(G_{s,p}f, \lambda)^{{\rm c}}}
			\1_{\real^{d} \setminus N}(X)
			\1_{\real^{d} \setminus N}(\widehat{X})
		\right]
		\notag\\
		&\leq
		K_{0}(s,p)^{q}
		{\mathbb E}
		\left[
			\left|
				X
				-
				\widehat{X}
			\right|^{qs}
			\{
				M(G_{s,p}f)(X)
				+
				M(G_{s,p}f)(\widehat{X})
			\}^{q}
			{\bf 1}_{\Omega(G_{s,p}f, \lambda)^{{\rm c}}}
		\right]\notag\\
		&\leq
		(2K_{0}(s,p))^{q}
		\lambda^{q}
		{\mathbb E}
		\left[
			\left|
				X
				-
				\widehat{X}
			\right|^{qs}
		\right].
	\end{align}
	By choosing $\lambda:=\e[|X-\widehat{X}|^{qs}]^{-\frac{1}{q+1-q/r}}$, we conclude the statement for $p_{X} \in L^{p^{*}}(\real^{d})$ from \eqref{eq:19} and \eqref{eq:20}.
\end{proof}

\section{Applications}\label{sec_3}
In this section, we apply Avikanen's estimates proved in Section \ref{sec_2} to numerical analysis on irregular functionals of a solution to stochastic differential equations (SDEs) based on the multilevel Monte Carlo method.
%, and to estimates of the $L^{2}$-time regularity of decoupled forward--backward stochastic differential equations (FBSDEs) with irregular terminal conditions.

\subsection{Upper bound and integrability of density functions}\label{sec_3_1}
In order to apply Avikanen's estimates proved in Section \ref{sec_2}, we need an appropriate upper bound or integrability of density functions.
In this subsection, we give some examples of random variables with a bounded or integrable density function, which are studied in various ways.

We first give the well-known fact as a conclusion of L\'evy's inversion formula.

\begin{Eg}\label{Eg_density_0}
	Let $X:\Omega \to \real^{d}$ be a random variable.
	If the characteristic function $\varphi_{X}(\xi):=\e[e^{\sqrt{-1}\langle \xi, X\rangle_{\real^{d}}}]$ belongs to $L^{1}(\real^{d})$, then by using L\'evy's inversion formula, $X$ admits a continuous density function $p_{X}$ of the form
	\begin{align*}
		p_{X}(x)
		=
		\frac{1}{(2\pi)^{d}}
		\int_{\real^{d}}
			e^{-\sqrt{-1}\langle x, \xi \rangle_{{\mathbb R}^{d}}}
			\varphi_{X}(\xi)
		\rd \xi
	\end{align*}
	(see, e.g., Proposition 2.5 in \cite{Sato} or Theorem 16.6 in \cite{Wi91}), and thus $X$ has a bounded density function.
\end{Eg}

Next, we recall the Gaussian two-sided bound for density functions of solutions to SDEs driven by a Brownian motion.

\begin{Eg}\label{Eg_GB_1}
	Let $B=(B(t))_{t \in [0,T]}$ be a $d$-dimensional standard Brownian motion,
	and let $X=(X(t))_{t \in [0,T]}$ be a solution to the following $d$-dimensional Markovian SDE of the form
	\begin{align}\label{SDE_0}
		\rd X(t)
		=
		b(t,X(t)) \rd t
		+
		\sigma(t,X(t))\rd B(t),
		~X(0)=x \in \real^{d},
		~t\in [0,T],
	\end{align}
	and let $X^{(n)}=(X^{(n)}(t))_{t \in [0,T]}$ be the Euler--Maruyama scheme for SDE \eqref{SDE_0} with time step $T/n$, which is defined by
	\begin{align*}
	\rd X^{(n)}(t)
	=
	b(\eta_{n}(t), X^{(n)}(\eta_{n}(t))) \rd t
	+
	\sigma(\eta_{n}(t), X^{(n)}(\eta _{n}(t))) \rd B(t),~
	X^{(n)}(0)=X(0),~
	t \in [0,T],
	\end{align*}
	where $\eta _{n}(s):=kT/n$ if $s \in [kT/n,(k+1)T/n)$, the drift coefficient $b:[0,T] \times \real^{d} \to \real^{d}$ and the diffusion matrix $\sigma:[0,T] \times \real^{d} \to \real^{d \times d}$ are measurable functions.
	Suppose that $b$ is bounded and $\sigma$ satisfies the following two conditions.
	\begin{itemize}
		\item[(i)]
		$a:=\sigma \sigma^{\top}$ is $\alpha$-H\"older continuous in space and $\alpha/2$-H\"older continuous in time for some $\alpha \in (0,1]$, that is,
		\begin{align*}
			\|a\|_{\alpha}
			:=
			\sup_{t \in [0,T], x \neq y}
			\frac{|a(t,x)-a(t,y)|}{|x-y|^{\alpha}}
			+
			\sup_{x\in \real^d, t \neq s}
			\frac{|a(t,x)-a(s,x)|}{|t-s|^{\alpha/2}}
			<\infty.
		\end{align*}
		
		\item[(ii)]
		The diffusion coefficient $\sigma$ is bounded and uniformly elliptic, that is, there exist $\underline{a}, \overline{a}>0$ such that for any $(t,x,\xi) \in [0,T] \times \real^d \times \real^d$, $\underline{a}|\xi|^2 \leq \langle a(t,x) \xi,\xi \rangle_{{\mathbb R}^{d}} \leq \overline{a} |\xi|^2$.
	\end{itemize}
	Then there exists a weak solution of SDE \eqref{SDE_0} and the uniqueness in law holds (see, Theorem 4.2, 5.6 in \cite{StVa69} or Proposition 1.14 in \cite{ChEn}), and it is well-known that for all $t \in (0,T]$, $X(t)$ admits a density function $p_{t}(x,\cdot)$ with respect to Lebesgue measure (e.g. Theorem 9.1.9 in \cite{StVa79}) which has the Gaussian two sided bound, that is, there exist $C_{\pm}>0$ and $c_{\pm}>0$ such that for any $(t,x,y) \in (0,T]\times \real^{d} \times \real^{d}$,
	\begin{align}\label{GB_1}
		C_{-} g_{c_{-}t}(x,y)
		&\leq
		p_{t}(x,y)
		\leq
		C_{+} g_{c_{+}t}(x,y).
	\end{align}
	As references of this two-sided bound, we refer to Theorem 9.4.2 in \cite{Fr64}, section 4.1, 4.2 and Remark 4.1 in \cite{LeMe10}, and \cite{Sh91} for time independent case (see also \cite{QiZh02,QiZh03,TaTa18} for a sharp two-sided bound in the case $\sigma \equiv I$).
	Therefore, it holds that $p_{t}(x,\cdot) \in L^{\infty}(\real^{d}) \cap L^{\Psi}(\real^{d}) \cap L^{p^{*}(\cdot)}(\real^{d}) $ for the complementary function $\Psi$ of an N-function $\Phi$ (see, Remark \ref{Rem_Orlicz_1} (ii) and Example \ref{Ex_Orlicz_0} (iii)) and $p^{*}(\cdot):=p(\cdot)/(p(\cdot)-1)$ for $p \in \mathcal{P}(\real^{d})$ with $1<p^{-}\leq p^{+} <\infty$.
	Moreover, for all $k=1,\ldots,n$, $X^{(n)}(kT/n)$ admits a density function $p_{kT/n}^{(n)}(x,\cdot)$ with respect to Lebesgue measure (e.g. section 2 in \cite{LeMe10}), which has the Gaussian two sided bound uniformly in $n$, that is, there exist $C_{\pm}>0$ and $c_{\pm}>0$ independent of $n$ such that for any $(k,x,y) \in \{1,\ldots,n\} \times \real^{d} \times \real^{d}$,
	\begin{align}\label{GB_2}
		C_{-} g_{c_{-}kT/n}(x,y)
		&\leq
		p_{kT/n}^{(n)}(x,y)
		\leq
		C_{+} g_{c_{+}kT/n}(x,y)
	\end{align}
	(see, Theorem 2.1 in \cite{LeMe10}).
	
	Moreover, the same or similar bounds \eqref{GB_1} hold for SDEs with a path--dependent or an unbounded drift (see, Theorem 2.5. in \cite{Ku17}, Theorem 3.4 in \cite{TaTa18} and Theorem 1.2 in \cite{MePeZh20}), and the Gaussian two sided bound \eqref{GB_1} holds for Brownian motions with a signed measure valued drift belonging to the Kato class $K_{d,1}$ (see, Theorem 3.14 in \cite{KiSo06}).
	
	It is also well-known that the Gaussian two--sided bound holds for the fundamental solution $\Gamma(s,x;t,y)$ of parabolic equations in the divergence form $(\frac{\partial}{\partial s}+\frac{1}{2}\sum_{i,j=1}^{d} \frac{\partial}{\partial x_{i}}a_{i,j}(x) \frac{\partial}{\partial x_{j}}) u(s,x)=0$ (see, \cite{Ar67}), and there exists a Hunt process with the transition density function $\Gamma$ (see, Example 4.5.2, Theorem A.2.2 in \cite{FOT} and Theorem I 9.4 in \cite{BlGe68}).
	
	On the other hand, Malliavin calculus can be used to study the regularity and upper bounds of density functions.
	%if the coefficients $b$ and $\sigma$ satisfy that $b \in C^{0,2+n}_{b}([0,T] \times \real^d;\real^d)$ and $\sigma \in C^{0,2+n}_{b}([0,T]\times \real^d;\real^{d\times d})$ for some $n \in \n \cup \{0\}$, and $\sigma$ is uniformly elliptic, then for all $t \in (0,T]$, $X(t)$ admits a density function $p_{t}(x,\cdot)$ with respect to Lebesgue measure which belongs to $C^{n}_{b}(\real^{d};\real)$ (see, \cite{KuSt82})
	Indeed, it is known that under H\"ormander's and the smoothness conditions on the coefficients, $X(t)$ admits a bounded and smooth density function (see, e.g., Theorem 6.16 in \cite{Shi04} and see, also \cite{DeMe10,KoMa13} for the Gaussian type estimates for density functions of solutions to degenerate SDEs).
	We also note that the Gaussian type two sided bound holds for density functions of solutions to SDEs driven by a fractional Brownian motion (see, \cite{BaNuOuTi16,BeKoTi16}).
\end{Eg}

The next example shows density estimates for path-dependent stochastic differential equations.

\begin{Eg}
	Let $B=(B(t))_{t \in [0,T]}$ be a $d$-dimensional standard Brownian motion,
	and let $X=(X(t))_{t \in [0,T]}$ be a solution to the following $d$-dimensional path--dependent SDE of the form
	\begin{align*}
		\rd X(t)
		=
		b(t,X) \rd t
		+
		\sigma(t,X)\rd B(t),
		~X(0)=x \in \real^{d},
		~t\in [0,T],
	\end{align*}
	where the drift coefficient $b:[0,T] \times C([0,T];\real^d) \to \real^{d}$ and the diffusion matrix $\sigma:[0,T] \times C([0,T];\real^d) \to \real^{d \times d}$ are measurable functions.
	\begin{itemize}
		\item[(i)]
		Suppose that the coefficients $b$ and $\sigma$ are continuous in time and bounded continuously G\^ateaux differentiable up to order $n+2$ in space, and $\sigma$ is uniformly elliptic.
		Then, by using Malliavin calculus, it is shown that for all $t \in (0,T]$, $X(t)$ admits a density function with respect to Lebesgue measure which belongs to $C^{n}_{b}(\real^{d};\real)$ (see, \cite{KuSt82}).
		
		\item[(ii)]
		Suppose that the coefficients $b$ and $\sigma$ are bounded, $\sigma$ is uniformly elliptic and there exist $\varepsilon>0$ and $C>0$ such that for any $(s,t,\omega) \in [0,T] \times [0,T] \times C([0,T];\real^d)$ with $s<t$, 
		\begin{align*}
			\sup_{1\leq j \leq d}
			|\sigma_{j}(t,\omega) -\sigma_j(s,\omega)|
			\leq
			C
			\left\{
				\log
				\left(
					\frac{1}{\sup_{s\leq u \leq t} |\omega_u-\omega_s|}
				\right)
			\right\}^{-(2+\varepsilon)},
		\end{align*}
		where $\sigma_{j}:=(\sigma_{1,j},\ldots,\sigma_{d,j})^{\top}$.
		Then, by using an interpolation method, it is shown that for all $t \in (0,T]$, $X(t)$ admits a density function $p_{t}(x,\cdot)$ with respect to Lebesgue measure which belongs to $L^{\mathbf{e}_{\log}}(\real^{d})$, where $\mathbf{e}_{\log}(x):=(1+|x|)\log (1+|x|)$ (see, Theorem 3.1 in \cite{BaCa}).
		
	\end{itemize}
\end{Eg}

Finally, we give examples for a two sided bound of density functions of solutions to SDEs driven by a rotation invariant $\alpha$-stable process.

\begin{Eg}\label{Eg_Levy}
	Let $Z=(Z(t))_{t \in [0,T]}$ be a rotation invariant $\alpha$-stable process in $\real^{d}$ with $\alpha \in (0,2)$ and $\e[e^{\sqrt{-1}\langle \xi,Z(t) \rangle}]=e^{-t|\xi|^{\alpha}}$, $\xi \in \real^{d}$(see Theorem 14.14 in \cite{Sato}), and let $X=(X(t))_{t \in [0,T]}$ be a solution to the following $d$-dimensional SDE of the form
	\begin{align}\label{SDE_stable_0}
		\rd X(t)
		=
		b(X(t)) \rd t
		+
		\sigma(X(t-))I\rd Z(t),~
		X(0)=x \in \real^{d},~
		t\in [0,T],
	\end{align}
	where $b:\real^{d} \to \real^{d}$ and $\sigma:\real^{d} \to \real$ are bounded measurable functions.
	Suppose that the drift coefficient $b$ is $\gamma$-H\"older continuous with $\gamma \in (0,1]$, and the jump intensity coefficient $a:=|\sigma|^{\alpha}$ is $\eta$-H\"older continuous with $\eta \in (0,1]$ and uniformly positive, that is, there exists $\underline{a}>0$ such that for any $x \in {\mathbb R}^{d}$, $a(x) \geq \underline{a}$.
	Under the balance condition $\alpha+\gamma>1$, Kulik \cite{Kul19} proved that the existence of a unique weak solution to the equation \eqref{SDE_stable_0}.
	Moreover, by using the parametrix method, he showed that for all $t \in (0,T]$, $X(t)$ admits a density function $p_{t}(x,\cdot)$ with respect to Lebesgue measure and gave its two sided bound, that is, there exist $C_{\pm}>0$ such that for any $(t,x,y) \in (0,T]\times \real^{d} \times \real^{d}$,
	\begin{align*}
		C_{-}
		\widetilde{p}_{t}(x,y)
		\leq
		p_{t}(x,y)
		\leq
		C_{+}
	\widetilde{p}_{t}(x,y),
	\end{align*}
	where
	\begin{align*}
		\widetilde{p}_{t}(x,y)
		:=
		\frac{1}{t^{d/\alpha} a(x)^{d/\alpha}}
		g^{(\alpha)}
		\left(
			\frac{y-v_{t}(x)}{t^{1/\alpha} a(x)^{1/\alpha}}
		\right),
	\end{align*}
	$g^{(\alpha)}$ is the density function of $Z(1)$ and $\{v_{t}(x)\}_{t \in [0,T]}$ is a solution to ODE $\rd v_{t}(x)=b(v_{t}(x)) \rd t$ with $v_{0}(x)=x$ (see Theorem 2.1 and Theorem 2.2 in \cite{Kul19}).
	Note that if $\gamma<1$, then such a solution of ODE may fail to be uniqueness, and if $\alpha+\gamma<1$, then a solution of SDE \eqref{SDE_stable_0} may fail to be uniqueness in law (see, Theorem 3.2 (ii) in \cite{TaTsuWa74}).
	
	Moreover, by the asymptotic behaviour of $g^{(\alpha)}$ (see, e.g., Theorem 2.1 in \cite{BlGe60}), we have $g^{(\alpha)}(x) \leq C \min\{1, |x|^{-d-\alpha}\}$ for some $C>0$, which implies that $p_{t}(x,\cdot) \in L^{\infty}(\real^{d}) \cap L^{\Psi}(\real^{d}) \cap L^{p^{*}(\cdot)}(\real^{d}) $ for the complementary function $\Psi$ of an N-function $\Phi$ (see, Remark \ref{Rem_Orlicz_1} (ii) and Example \ref{Ex_Orlicz_0} (iii)) and $p^{*}(\cdot):=p(\cdot)/(p(\cdot)-1)$ for $p \in \mathcal{P}(\real^{d})$ with $1<p^{-}\leq p^{+} <\infty$, (see, also \cite{KaSz15,KnSc13,Kuh19} for upper bounds of density functions of L\'evy processes).
\end{Eg}

\subsection{Multilevel Monte Carlo method}\label{sec_3_2}
In this subsection, we apply Avikainen's estimates proved in Section \ref{sec_2} to the multilevel Monte Carlo method for solutions to SDE \eqref{SDE_0}.
We first define the union of function spaces $F(\real^{d})$ by
\begin{align*}
	F(\real^{d})
	:=
	\left\{
		BV(\real^{d})
		\cup
		W^{1,\Phi}(\real^{d})
		\cup
		W^{1,p(\cdot)}(\real^{d})
		\cup
		W^{s,p}(\real^{d})
	\right\}
	\bigcap
	L^{\infty}(\real^{d})
\end{align*}
for an N-function $\Phi$ with its complementary function $\Psi$ which satisfies the $\Delta_{2}$-condition, a variable exponent $p(\cdot) \in \mathcal{P}^{\log}(\real^{d})$ with $1<p^{-}\leq p^{+} <\infty$ and $(s,p) \in (0,1] \times [1,\infty)$.

We consider the computational complexity of the mean squared error (MSE) to estimate the expectation of $P:=f(X(T))$ for some measurable function $f:\real^{d} \to \real$ with $\e[|f(X(T))|]<\infty$, by using the standard and multilevel Monte Carlo method.

We first recall the standard Monte Carlo method.
Let $X^{(h)}=(X^{(h)}(t))_{t \in [0,T]}$ be the Euler--Maruyama scheme for SDE \eqref{SDE_0} with time step $h \in (0,T)$, which is defined by
\begin{align*}
	\rd X^{(h)}(t)
	=
	b(\eta_{h}(t), X^{(h)}(\eta_{h}(t))) \rd t
	+
	\sigma(\eta_{h}(t), X^{(h)}(\eta _{h}(t))) \rd B(t),~
	X^{(h)}(0)=X(0),~
	t \in [0,T],
\end{align*}
where $\eta _{h}(s):=kh$ and $k$ is the natural number such that $s/h-1<k\leq s/h$.
We define $\widehat{P}^{(h)}:=f(X^{(h)}(T))$.
Let $\widehat{Y}^{(h)}$ be an estimator for $\widehat{P}^{(h)}$.
For example, one may use $\widehat{Y}^{(h)}$ as the arithmetic mean, that is,
\begin{align*}%\label{eq:41}
	\widehat{Y}^{(h)}
	:=
	N^{-1}
	\sum_{i=1}^{N}
	\widehat{P}^{(h,i)},
\end{align*}
where $\widehat{P}^{(h,1)}, \ldots, \widehat{P}^{(h,N)}$ are i.i.d. random variables which have the same distribution as $\widehat{P}^{(h)}$.
We suppose that the weak rate of convergence for $X^{(h)}$ is $\alpha>0$, that is, there exists $c_{0}>0$ such that $|\e[f(X(T))]-\e[f(X^{(h)}(T))]| \leq c_{0}h^{\alpha}$.
Then the mean squared error is estimated as follows:
\begin{align*}
	\text{MSE}
	:=
	\e[|\widehat{Y}^{(h)}-\e[P]|^{2}]
	=
	\e[|\widehat{Y}^{(h)}
	-
	\e[\widehat{Y}^{(h)}]|^{2}]
	+
	\left|\e[\widehat{P}^{(h)}]-\e[P]\right|^{2}
	\leq
		\mathrm{Var}[\widehat{P}^{(h)}]
		N^{-1}
		+
		c_{0}^{2}
		h^{2\alpha}.
\end{align*}
Assume that $\sup_{h \in (0,T)} \mathrm{Var}(\widehat{P}^{(h)})<\infty$.
Then, if we would like to make $\mathrm{MSE} \leq \varepsilon^{2}$ for a given $\varepsilon>0$, we choose $N$ and $h$ to satisfy $\sup_{h \in (0,T)} \mathrm{Var}[\widehat{P}^{(h)}]N^{-1} \leq \varepsilon^{2}/2$ and $c_{0}^{2}h^{2\alpha} \leq \varepsilon^{2}/2$.
Then the computational complexity for $\widehat{Y}^{(h)}$ is estimated above by $c_{1}\varepsilon^{-(2+1/\alpha)}$ for some constant $c_{1}>0$.

Now we recall the multilevel Monte Carlo method.
Let $M \in {\mathbb N}$ and $X_{\ell}=(X_{\ell}(t))_{t \in [0,T]}$, $\ell=0,\ldots,L$ be numerical approximations to $X$ with each time step $h_{\ell}:=T/M^{\ell}$ and define $\widehat{P}_{\ell}:=f(X_{\ell}(T))$.
Then it holds that
\begin{align*}
	\e[\widehat{P}_{L}]
	=
	\sum_{\ell=0}^{L}
	\e[
		\widehat{P}_{\ell}
		-
		\widehat{P}_{\ell-1}
	],
\end{align*}
where $\widehat{P}_{-1}:=0$.
Let $\widehat{Y}_{\ell}$, $\ell=0,\ldots,L$ be independent estimators for each $\e[\widehat{P}_{\ell}-\widehat{P}_{\ell-1}]$ and $C_{\ell}$, $\ell=0,\ldots,L$ be their  corresponding computational complexities.
For example, one may use $\widehat{Y}_{\ell}$ as the arithmetic mean, that is,
\begin{align}\label{MLMC_estimator_0}
	\widehat{Y}_{\ell}
	=
		N_{\ell}^{-1}
	\sum_{i=1}^{N_{\ell}}
	\left(
		\widehat{P}_{\ell}^{(\ell,i)}
		-
		\widehat{P}_{\ell-1}^{(\ell,i)}
	\right),
\end{align}
where $\widehat{P}_{\ell}^{(\ell,1)}-\widehat{P}_{\ell-1}^{(\ell,1)},\ldots,\widehat{P}_{\ell}^{(\ell,N_{\ell})}-\widehat{P}_{\ell-1}^{(\ell,N_{\ell})}$ are i.i.d random variables which have the same distribution as $\widehat{P}_{\ell}-\widehat{P}_{\ell-1}$.
Note that the random variable $\widehat{P}_{\ell}-\widehat{P}_{\ell-1}$ is the deference between two discrete approximations with different time steps $h_{\ell}$ and $h_{\ell-1}$, and the key point is that they are defined by the same Brownian motion.
We define the estimator $\widehat{Y}:=\sum_{\ell=0}^{L} \widehat{Y}_{\ell}$ and its computational complexity $C_{\text{MLMC}}:=\sum_{\ell=0}^{L}C_{\ell}$.
Then the following complexity theorem holds for the MLMC method.

\begin{Thm}[Complexity theorem, Theorem 3.1 in \cite{Gi08}, Theorem 2.1 \cite{Gi15}]\label{MLMC_0}
	Let $X$ be a solution to SDE \eqref{SDE_0}, and define $P:=f(X(T))$.
	We assume that $\{\widehat{P}_{\ell}\}_{\ell=0,\ldots,L}$, independent estimators $\{\widehat{Y}_{\ell}\}_{\ell=0,\ldots,L}$ and their  corresponding computational complexities $\{C_{\ell}\}_{\ell=0\ldots, L}$ satisfy the following conditions:
	there exist positive constants $c_{1}, c_{2}, c_{3}$ and $\alpha,\beta$ such that $\alpha \geq \beta/2$ and
	(i) $|\e[P]-\e[\widehat{P}_{\ell}]|\leq c_{1}h_{\ell}^{\alpha}$;
	(ii) $\e[\widehat{Y}_{\ell}]=\e[\widehat{P}_{\ell}-\widehat{P}_{\ell-1}]$;
	(iii) $\mathrm{Var}[\widehat{Y}_{\ell}] \leq c_{2} N_{\ell}^{-1} h_{\ell}^{\beta}$;
	(iv) $C_{\ell} \leq c_{3}N_{\ell} h_{\ell}^{-1}$.
	Then for any $\varepsilon \in (0,1/e)$, the mean squared error is estimated by
	\begin{align*}
		\text{MSE}
		:=
		\e\left[
			\left|
				\widehat{Y}
				-
				\e[P]
			\right|^2
		\right]
		\leq
		\varepsilon^2
	\end{align*}
	with the computational complexity $C_{\text{MLMC}}$ for $\widehat{Y}$ bounded by
	\begin{align*}
		C_{\text{MLMC}}
		\leq
		\left\{ \begin{array}{ll}
		\displaystyle
			c_{4}\varepsilon^{-2},
			&\text{ if } \beta \in (1,\infty),\\
		\displaystyle
			c_{4}\varepsilon^{-2} (\log \varepsilon)^{2},
			&\text{ if } \beta =1,  \\
		\displaystyle
			c_{4}
			\varepsilon^{-\{2+(1-\beta)/\alpha\}},
			&\text{ if } \beta \in (0,1)
		\end{array}\right.
	\end{align*}
	for some constants $c_{4}>0$.
\end{Thm}

Before applying Theorem \ref{MLMC_0} to the Euler--Maruyama scheme $X^{(n)}$ with time step $h=T/n$, we provide the strong rate of convergence for \eqref{MSE_0} for irregular functions $f \in F(\real^{d})$.

\begin{Thm}\label{Cor_0}
	Suppose that the coefficients $b$ and $\sigma$ of SDE \eqref{SDE_0} are bounded and Lipschitz continuous in space, and $1/2$-H\"older continuous in time, and $\sigma$ is uniformly elliptic.
	Then for any $f \in F(\real^{d})$, $q \in [1,\infty)$ and $\delta \in (0,1)$, there exist $C_{\text{EM}}(q,\delta)>0$ and $C_{\text{EM}}(q)>0$ such that
	\begin{align*}
		\e\left[
			\left|
				f(X(T))
				-
				f(X^{(n)}(T))
			\right|^{q}
		\right]
		\leq
		\left\{ \begin{array}{ll}
			\displaystyle
				C_{\text{EM}}(q,\delta)
				n^{-\frac{\delta}{2}},
				&\text{ if } f \in BV(\real^{d}) \cap L^{\infty}(\real^{d}),  \\
			\displaystyle
				C_{\text{EM}}(q)
				n^{-\frac{q}{2(q+1)}},
				&\text{ if } f \in \{W^{1,\Phi}(\real^{d}) \cup W^{1,p(\cdot)}(\real^{d})\} \cap L^{\infty}(\real^{d}),\\
			\displaystyle
				C_{\text{EM}}(q)
				n^{-\frac{pqs}{2(q+p)}},
				&\text{ if } f \in W^{s,p}(\real^{d}) \cap L^{\infty}(\real^{d}).
		\end{array}\right.
	\end{align*}
\end{Thm}
\begin{proof}
	We first note that under the assumptions on the coefficients, the strong rate of convergence for the Euler--Maruyama scheme is $1/2$, that is, for any $p >0$, there exists a constant $C_{p}>0$ such that $\e[|X(T)-X^{(n)}(T)|^{p}]^{1/p} \leq C_{p} n^{-1/2}$ (see, e.g. \cite{KP}).
	It follows from Example \ref{Eg_GB_1} that $X(T)$ and $X^{(h)}(T)$ admit density functions $p_{T}(x,\cdot)$ and $p_{T}^{(n)}(x,\cdot)$ with respect to Lebesgue measure which have the Gaussian upper bound \eqref{GB_1} and \eqref{GB_2}, respectively.
	Hence we have $p_{T}(x,\cdot), p_{T}^{(n)}(x,\cdot) \in L^{\infty}(\real^{d}) \cap L^{\Psi}(\real^{d}) \cap L^{p^{*}(\cdot)}(\real^{d}) $.
	Therefore, by using Theorem \ref{main_0}, \ref{main_1}, \ref{main_2}, \ref{main_3} with $r=\infty$, we conclude the statement.
\end{proof}

\begin{Rem}\label{Rem_GB_0}
	\begin{itemize}
		\item[(i)]
		Recently under non-Lipschitz coefficients, the strong rate of convergence for the Euler--Maruyama scheme are widely studied (see, \cite{BaHuYu19,BuDaGe19,GyRa11,KuSc19,LeSz17b,MeTa,MuYa20,NT1,NT2}).
		
		\item[(ii)]
		%Let $n \in \n$ and $h=T/n$.
		If we assume $f \in L^{r}(\real^{d},p_{T}(x,\cdot))$ for some $r \in (1,\infty)$, then similar estimates of Theorem \ref{Cor_0} hold.
		Note that the constants $C_{EM}$ may depend on $\|f\|_{L^{r}(\real^{d},p_{T}^{(n)}(x,\cdot))}$, especially on the time step $h=T/n$ (see also Remark \ref{rem:2.12} (v) and Remark \ref{rem:2.13} (iii)).
		However, by using the Gaussian upper bound \eqref{GB_2} for the density of $X^{(n)}(T)$, under the additional assumption $f \in L^{r}(\real^{d}, g_{c_{+}T}(x,\cdot))$, the constants $C_{EM}$ are uniformly bounded with respect to the time step $h=T/n$.
	\end{itemize}
\end{Rem}

As applications of Theorem \ref{MLMC_0} and Theorem \ref{Cor_0}, we have the following two examples for irregular functions $f \in F(\real^{d})$.

\begin{Eg}\label{Ex_MLMC_0}
	Let $X_{\ell}$ be the Euler--Maruyama scheme with time step $h_{\ell}=T/M^{\ell} $ and $\widehat{Y}_{\ell}$ be the independent estimator defined by \eqref{MLMC_estimator_0}.
	Suppose the coefficients $b$ and $\sigma$ of SDE \eqref{SDE_0} satisfy $b \in C^{1,3}_{b}([0,T]\times \real^{d};{\mathbb R}^{d})$, $\sigma \in C^{1,3}_{b}([0,T]\times \real^{d};{\mathbb R}^{d \times d})$ and $\partial_{t}\sigma \in C^{0,1}_{b}([0,T]\times \real^{d};{\mathbb R}^{d \times d})$, and $\sigma$ is uniformly elliptic.
	Then it follows from Theorem 2.5 in \cite{GoLa08} that for any bounded measurable function $f:\real^{d} \to \real$, the weak rate of convergence $\alpha$ in Theorem \ref{MLMC_0} is one (see also, Theorem 3.5 in \cite{BaTa96}, Corollary 22 in \cite{Gu06}, Theorem 1.1 in \cite{KoMa02} and Theorem 1 in \cite{TaTu90}).
	Note that
	\begin{align*}
		\mathrm{Var}[\widehat{Y}_{\ell}]
		&=
		\e[|\widehat{Y}_{\ell}-\e[\widehat{Y}_{\ell}]|^{2}]
		=
		N_{\ell}^{-1}
		\mathrm{Var}[\widehat{P}_{\ell}-\widehat{P}_{\ell-1}]
	\end{align*}
	for each $\ell=0,\ldots,L$.
	Hence by using Theorem \ref{Cor_0}, for  any $f \in F(\real^{d})$, we have
	\begin{align*}
		&\mathrm{Var}[\widehat{P}_{\ell}-\widehat{P}_{\ell-1}]
		\leq
		\left|
			\e[\widehat{P}_{\ell}]
			-
			\e[\widehat{P}_{\ell-1}]
		\right|^{2}
		+
		2\e[|P-\widehat{P}_{\ell}|^{2}]
		+
		2\e[|P-\widehat{P}_{\ell-1}|^{2}]
		\\&\leq
		c_{0}^{2}
		\{
			h_{\ell}
			+
			h_{\ell-1}
		\}^{2}
		+
		\left\{ \begin{array}{ll}
		\displaystyle
			2C_{\text{EM}}(2,\delta)
			\{
				h_{\ell}^{\frac{\delta}{2}}
				+
				h_{\ell-1}^{\frac{\delta}{2}}
			\},
			&\text{ if } f \in BV(\real^{d}) \cap L^{\infty}(\real^{d}),\\
		\displaystyle
			2C_{\text{EM}}(2)
			\{
				h_{\ell}^{\frac{1}{3}}
				+
				h_{\ell-1}^{\frac{1}{3}}
			\},
			&\text{ if } f \in \{W^{1,\Phi}(\real^{d}) \cup W^{1,p(\cdot)}(\real^{d})\} \cap L^{\infty}(\real^{d}),\\
		\displaystyle
			2C_{\text{EM}}(2)
			\{
				h_{\ell}^{\frac{ps}{p+2}}
				+
				h_{\ell-1}^{\frac{ps}{p+2}}
			\},
			&\text{ if } f \in W^{s,p}(\real^{d}) \cap L^{\infty}(\real^{d}).
		\end{array}\right.
	\end{align*}
	Hence, for $f \in F(\real^{d})$, the computational complexity $C_{\text{MLMC}}$ for $\widehat{Y}$ is bounded from above by
	\begin{align*}
		C_{\text{MLMC}}
		\leq
		\left\{ \begin{array}{ll}
		\displaystyle
			c_{4}
			\varepsilon^{-\frac{6-\delta}{2}},
			&\text{ if } f \in BV(\real^{d}) \cap L^{\infty}(\real^{d}),  \\
		\displaystyle
			c_{4}
			\varepsilon^{-\frac{8}{3}},
			&\text{ if } f \in \{W^{1,\Phi}(\real^{d}) \cup W^{1,p(\cdot)}(\real^{d})\} \cap L^{\infty}(\real^{d}),\\
		\displaystyle
			c_{4}
			\varepsilon^{-(3-\frac{ps}{p+2})},
			&\text{ if } f \in W^{s,p}(\real^{d}) \cap L^{\infty}(\real^{d}).
		\end{array}\right.
	\end{align*}
\end{Eg}

\begin{Eg}\label{Ex_MLMC_1}
	Let $X_{\ell}$ be the Euler--Maruyama scheme with time step $h_{\ell}=T/M^{\ell}$ and $\widehat{Y}_{\ell}$ be the independent estimator defined by \eqref{MLMC_estimator_0}.
	Suppose the coefficients $b$ and $\sigma$ of SDE \eqref{SDE_0} are bounded and Lipschitz continuous in space, and $1/2$-H\"older continuous in time, and $\sigma$ is uniformly elliptic.
	Then for any bounded measurable function $f:\real^{d} \to \real$ and $\delta \in (0,1)$, the weak rate of convergence $\alpha$ in Theorem \ref{MLMC_0} is $\delta/2$, (see Theorem 1.1 in \cite{KoMe17}).
	Therefore, for $f \in F(\real^{d})$, the computational complexity $C_{\text{MLMC}}$ for $\widehat{Y}$ is bounded from above by
	\begin{align*}
		C_{\text{MLMC}}
		\leq
		\left\{ \begin{array}{ll}
		\displaystyle
			c_{4}
			\varepsilon^{-1-\frac{2}{\delta}},
			&\text{ if } f \in BV(\real^{d}) \cap L^{\infty}(\real^{d}),  \\
		\displaystyle
			c_{4}
			\varepsilon^{-2-\frac{4}{3\delta}},
			&\text{ if } f \in \{W^{1,\Phi}(\real^{d}) \cup W^{1,p(\cdot)}(\real^{d})\} \cap L^{\infty}(\real^{d}),\\
		\displaystyle
			c_{4}
			\varepsilon^{-2-\frac{p(1-s)+2}{\delta(p+2)}},
			&\text{ if } f \in W^{s,p}(\real^{d}) \cap L^{\infty}(\real^{d}).
		\end{array}\right.
	\end{align*}
\end{Eg}

\subsection*{Acknowledgements}
The authors would like to thank the anonymous referees for their careful readings and valuable comments.
The authors deeply grateful to Professor Flavien Leger for his valuable comments.
The first author was supported by JSPS KAKENHI Grant Number 19K14552.
The second author was supported by Sumitomo Mitsui Banking Corporation.
The third author was supported by JSPS KAKENHI Grant Number 17J05514.

\end{document}